\documentclass[11pt]{article}
\usepackage{setspace}

\usepackage{amsfonts,amsmath,amssymb,mathrsfs,amsthm} 
\usepackage[top=25mm, bottom=25mm, left=25mm, right=25mm]{geometry} 
\usepackage{graphicx} 
\usepackage[nodayofweek]{datetime}			
\usepackage[authoryear,round]{natbib}
\usepackage{algorithm}
\usepackage{algpseudocode}
\usepackage{tkz-graph}
\usetikzlibrary{arrows}
\usepackage{xcolor,colortbl}
\usepackage{caption}
\usepackage{subcaption}
\usepackage{eso-pic}
\usepackage{booktabs}
\usepackage{bm}
\usepackage{tikz-cd}
\usepackage{tocbasic}
\usepackage{multirow}
\usepackage{pdflscape}
\DeclareTOCStyleEntry[dynnumwidth]{tocline}{figure}
\DeclareTOCStyleEntry[dynnumwidth]{tocline}{table}

\newtheorem{definition}{Definition}[section]
\newtheorem{theorem}{Theorem}[section]
\newtheorem{lemma}{Lemma}[section]
\newtheorem{proposition}{Proposition}[section]
\newtheorem{corollary}{Corollary}[section]
\newtheorem{example}{Example}[section]
\newtheorem{remark}{Remark}[section]

\newtheoremstyle{mystyle}  
  {}   				
  {}   				
  {}  				
  {0pt}       		
  {\bfseries}	
  {:}        
  {5pt plus 1pt minus 1pt}				  
  {}          

\theoremstyle{mystyle} 
\newtheorem*{eg}{Working Example}

\makeatletter
\def\degrees#1{\gdef\@degrees{#1}}
\def\@degrees{\@latex@error{No \noexpand\degrees given}\@ehc}
\makeatother

\usepackage[figuresleft]{rotating}
\usepackage{enumerate}
\usepackage{hyperref}
\hypersetup{
    linkcolor=blue,
    citecolor=blue,
    filecolor=magenta,      
    colorlinks=true, 
    urlcolor=cyan,
    pdftitle={},    pdfpagemode=FullScreen
    }

\newcommand{\lcurl}{\left\{}
\newcommand{\rcurl}{\right\}}
\newcommand{\PF}{\text{PF}}

\newcommand{\Th}{\textsuperscript{th}}	

\newcommand{\ZZ}{\mathbb{Z}}	
\newcommand{\RR}{\mathbb{R}}    
\newcommand{\II}{\mathbb{I}}

\newcommand{\PP}{\mathbb{P}}		
\newcommand{\EE}{\mathbb{E}}		

\usepackage{mathtools}
\mathtoolsset{showonlyrefs}
\usepackage{authblk}
\usepackage{pdfpages}
\makeatletter
\renewcommand\AB@affilsepx{, \protect\Affilfont}
\makeatother

\begin{document}
\title{New Modelling and Solution Frameworks for Multiple Objectives, Decomposability and Impulsivity in Continuous-time Markov Decision Processes}
\author[1]{Luke Fairley}
\author[1]{Rob Shone}
\author[1]{Peter Jacko}
\author[2]{Jefferson Huang}
\affil[1]{Lancaster University}\affil[2]{Naval Postgraduate School}
\date{\today}

\maketitle

\begin{abstract}
    Impulsive control, in the context of continuous-time Markov decision processes (CTMDPs), refers to actions that instantaneously update the state of the process, leading to theoretical complications. Previous work has overcome these complications, but at the expense of approachability and ease of implementation: those models are mathematically complex and incompatible with pre-existing solution methods. This paper presents an alternative formulation of impulsive control for CTMDPs that is easier to employ in the modelling of real-world problems and can be solved by pre-existing solution methods. This formulation is then extended to allow for multiple objectives and decomposability. A decomposition-based heuristic method is developed that leverages Dantzig-Wolfe decomposition to find approximate Pareto fronts for large multi-objective decomposable impulsive CTMDPs. The modelling framework and solution methodology are applied to a dynamic maintenance problem to demonstrate their computational efficiency, solution space coverage, and solution quality.
\end{abstract}

\section{Introduction}
Sequential decision making under uncertainty is required in many industries including finance, medicine, computing and engineering. For example, in maintenance, decisions are made as to which components of a system need to be repaired, the extent of the repairs, and when these should be performed \citep{papakonstantinou2014MaintenancePOMDP}. These decisions are not made at a single point in time, but instead are revisited at many different stages as the system continues to degrade in an unpredictable manner due to wear and tear, or external shocks. Another example is inventory control, in which demands for items vary randomly from day to day, and decisions are made as to when new orders should be placed and how many units should be ordered. Additional uncertainty may arise from unreliable supply chains or uncertain lead times \citep{boute2022deepRLInventory}. A third example is the control of queueing systems, where there is uncertainty in customer inter-arrival times and service times. Decisions may involve the routing of customers to different queues, or the acceptance or rejection of customers from the system entirely \citep{shone2020queueing,anselmi2024admissionControl, chen2024queueing}. 

A common framework for modelling sequential decision problems under uncertainty is the Markov decision process (MDP). In general, an MDP is an extension of a Markov chain which allows actions to be taken by a decision-maker. These actions affect both immediate costs (or, alternatively, rewards) and state transitions. Rather than simply choosing actions that minimise immediate costs, the goal is instead to minimise some measure of long-run cost by controlling the system in an optimal manner. As such, MDPs can capture much of the structure of many sequential decision problems under uncertainty. 

In this paper we focus on continuous-time MDPs (CTMDPs) and consider variations thereof which are seldom seen together in the academic literature despite being commonly present in real-life situations, including those with impulsive actions, multiple objectives, and a decomposable structure. By developing and analysing the necessary modelling frameworks, and proposing a heuristic solution methodology, we then investigate a CTMDP model for the dynamic maintenance of series-parallel systems.

\subsection{Impulsive Actions}
One variation of the CTMDP involves the use of \emph{impulsive actions}. These are actions that cause an instantaneous change in the state, as opposed to the more common \emph{gradual} actions which mandate a random non-zero sojourn time before the transition takes place. Such a modelling feature is useful for any problem where the state changes immediately in response to the action. A pioneering theoretical treatment of impulsive control or gradual-impulsive control (for systems that allow for both types of action) is given by \cite{dufour2015impulsive}. However, we believe this framework to be mathematically complex in such a way as to limit its accessibility and potential applicability, leaving room for the development of formulations which capture similar behaviour whilst remaining tractable for real-world modelling and solution.

\subsection{Multiple Objectives}
The second variation of the MDP we consider is the multi-objective MDP, or MOMDP. In a standard single-objective MDP, one aims to control the system in such a way that the long-run scalar costs are minimised. In an MOMDP, one aims to strike a balance between elements of a vector cost function in the long run, and to find multiple ways to control the system to strike different balances. These different costs can be diametrically opposed, so it is necessary to find a range of compromises. A survey on MOMDPs is given by \cite{roijers2013momdp_survey}. The literature on multi-objective MDPs is relatively sparse; some investigations of such problems include \cite{Wakuta2000momdp} and \cite{GHOSH1990momdp}. We argue that many common problems which are modelled as MDPs lend themselves to multi-objective modelling and should be treated as such. To give a few examples: in maintenance problems, natural objectives can be based on  maintenance costs and system reliability or performance; in inventory problems, objectives can include holding cost minimisation and demand satisfaction; and in queueing control, objectives can be based on costs and customer satisfaction, where the latter may depend, for example, on the time spent in the queue. These problems, and many more, require a balance between monetary reward and some measurable outcome which may not scale linearly with --- or could be entirely separate from --- the monetary outcome.

\subsection{Decomposability}
The third variation we consider is that of decomposability, meaning an MDP that can be thought of as being comprised of multiple smaller MDPs, usually with an objective function or constraints linking them together. A common example is that of the weakly-coupled MDP, which is made up of multiple smaller MDPs with their own state spaces, action spaces, and stochastic dynamics, but with global constraints on action sets. For example, costs may be associated with actions taken in each small MDP, and there may be a global budget over all the MDPs for each decision epoch. Explorations of weakly-coupled MDPs are given by \cite{meuleau1998weaklyCoupled}. \cite{bertsimas2016decomp} introduces a broader definition of a decomposable discrete-time MDP which includes weakly-coupled MDPs, but still limits this to discrete time. In this work, the decomposed MDPs will be linked together by objective functions which may be non-linear, rather than constraints.

{\color{black}
\subsection{Paper Outline}

\autoref{paper2:model:ctmdp} and \autoref{paper2:model:impulsive} first formalise and analyse an alternative formulation of impulsive control (i.e. instant transitions); notably, we use equivalence relations to analyse this framework, and prove that a particular class of intuitive policies always contains an optimal solution. We also demonstrate that it is always possible to obtain such a policy from an optimal policy not in this class, if one exists. \autoref{paper2:model:decomp} follows by extending this impulsive CTMDP model to the \emph{decomposable} impulsive CTMDP, a model which can be thought of as multiple separable impulsive CTMDPs which can be controlled together, and which are unified by shared objective functions. Further analysis shows that, when controlled by a specific type of policy, the resulting high-dimensional Markov process decomposes into smaller independent Markov processes, up to an equivalence relation. \autoref{paper2:solution} then presents a block-diagonalisation of the LP formulation of a decomposable impulsive CTMDP under certain restrictions, where this block-diagonal form can then by solved via Dantzig-Wolfe decomposition. This in turn allows us to develop a heuristic in which we find Pareto fronts for the compartment CTMDPs that made up the decomposable impulsive CTMDP, and then piece these solutions back together with a bi-objective binary program.  In \autoref{paper2:app} we apply our modelling framework and solution methodology to a dynamic maintenance problem motivated by reliability engineering, and investigate the strengths and weaknesses of our solution methodology compared to an exact LP solver. \autoref{paper2:conc} then concludes with a discussion of strengths and limitations, and with suggestions for further research. Throughout, most of the mathematical proofs are deferred to the appendices, with some alternative direct proofs available in the online supplementary materials, and shorter proofs left in the main body. 
}

\subsection{Our Contributions}
The contributions of this paper are as follows:
\begin{itemize}
    \item We provide an alternative implementation of impulsive control in CTMDPs that we argue is more approachable for modellers, and more compatible with standard solution methods. Specifically, we focus on finite state and action spaces and therefore avoid the need for technical considerations such as Borel spaces and measurability.
    \item We extend our modelling framework for impulsive control to include multiple objectives and decomposability. Specifically, we introduce a \emph{decomposable impulsive CTMDP}, which can be thought of as a CTMDP composed of multiple separate impulsive CTMDPs. We discuss the subtleties of this and develop the necessary structure to ensure that whole CTMDP and smaller compartment CTMDPs are both truly impulsive.
    \item We develop a heuristic solution methodology for multi-objective decomposable impulsive CTMDPs. We prove that a restricted version of the problem of finding the Pareto front for a decomposable impulsive CTMDP can be transformed to have a block-diagonal structure, which enables the use of Dantzig-Wolfe decomposition \citep{dantzig1960decomposition}. In short, we show that with some additional assumptions on the solution space, finding an approximate Pareto front of the decomposable impulsive CTMDPs is equivalent to finding the Pareto fronts of the compartment CTMDPs and piecing these solutions back together. 
\end{itemize}

\subsection{Preliminaries}
This paper assumes familiarity with continuous-time Markov decision processes, multi-objective optimisation, and Dantzig-Wolfe decomposition. Additionally, the theoretical work makes extensive use of equivalence relations, for which we introduce some notation and terminology here. If $\sim$ is an equivalence relation, we write $s\sim t$ if $s$ is equivalent to $t$ under $\sim$. The \emph{equivalence class} of $s$ can be denoted by $\widehat{s}$. If $\sim$ is an equivalence class over a set $\mathcal{S}$, we denote the set of equivalence classes by $\widehat{\mathcal{S}}$. We call the function $\varphi: \mathcal{S}\to\widehat{\mathcal{S}}$ that maps every element $s$ of the base set $\mathcal{S}$ to its equivalence class $\widehat{s}$ in $\widehat{\mathcal{S}}$ the \emph{canonical surjection}. To compare the ``strength'' of equivalence relations, for two equivalence relations $\sim$ and $\sim'$ we say that $\sim$ \emph{as strong as} $\sim'$ if and only if $s\sim t \implies s \sim' t$ for all $s,t\in \mathcal{S}$. Some literature refers to this as $\sim$ being ``finer'' than $\sim'$.

\section{Modelling Framework}
\label{paper2:model}
This section introduces a succession of modelling frameworks which build upon each other. In \autoref{paper2:model:ctmdp} we introduce a notion of CTMDPs that allows for dummy ``self-transitions''. This allows for the development of a range of equivalence classes that will later be useful for the analysis of our impulsive and decomposable frameworks. \autoref{paper2:model:impulsive} introduces our formulation of CTMDPs with impulsive control, building upon the CTMDPs with self-transitions. \autoref{paper2:model:decomp} then extends this to decomposable impulsive CTMDPs, formalising the notion of multiple stochastically independent impulsive CTMDPs being treated as one impulsive CTMDP with shared objective functions. Each of Sections 
\ref{paper2:model:ctmdp}, \ref{paper2:model:impulsive} and \ref{paper2:model:decomp} are further subdivided into a ``Model'' subsection to give a softer introduction, followed by a ``Theory'' subsection to provide a more rigorous analysis. As such, the reader who is interested only in the modelling frameworks, their interpretations, and how to use them may choose to focus on the first part of each subsection.

\subsection{CTMDPs with Self-Transitions}
\label{paper2:model:ctmdp}
\subsubsection{Model}
\label{paper2:model:ctmdp:model}
A finite-state and finite-action CTMDP is typically defined as a 4-tuple:$$\mathfrak{C} = \left\langle \mathcal{S}, (\mathcal{A}(s))_{s\in\mathcal{S}}, q:\mathcal{SA}\times\mathcal{S}\to\RR, \bm{c}:\mathcal{SA}\to \RR^k \right\rangle,$$
where $\mathcal{S}$ is the state space, $\mathcal{A}(s)$ is the feasible action space for state $s\in\mathcal{S}$, $\mathcal{SA}=\left\{(s,a):s\in\mathcal{S},a\in\mathcal{A}(s)\right\}$ is the set of feasible state-action pairs, $q(s,a,s')$ is the transition rate from state $s$ to state $s'$ when taking action $a$, and $\bm{c}(s,a)$ is the vector cost rate when taking action $a$ in state $s$. We write $\mathcal{A}:=\bigcup_{s\in\mathcal{S}}\mathcal{A}(s)$ as the space of all possible actions. A \emph{deterministic stationary} (DS) policy is a map $\mu:\mathcal{S}\to\mathcal{A}$. Throughout, we assume that our CTMDPs are \emph{weakly-communicating}, meaning the state space can be split into two sets: one for which for any two states $s,s'$ in the set, there exists a DS policy such that $s'$ can be reached from $s$ in finite time with non-zero probability; and another set of states for which all are transient under any DS policy. A stronger condition is \emph{unichain}, where for any DS policy, the induced Markov process is also unichain (i.e. has one recurrent class and possibly some transient classes).

Under the usual definition of transition rates, we have $q(s,a,s) = -\sum_{s'\in\mathcal{S}:s'\neq s}q(s,a,s')$; hence, we do not generally consider a state to be able to transition to itself. This is because, even if one could consider some event with an exponentially distributed time that would technically cause state $s$ to transition to itself, this stochastic information is lost in the $q$ function, as the positive ``outbound'' and negative ``inbound'' transition rates of a self-transition would cancel out. However, we argue that it can be useful to keep this information explicitly encoded. To this end we define the \emph{CTMDP with self-transitions} as a 4-tuple:$$\mathfrak{C} = \left\langle \mathcal{S}, (\mathcal{A}(s))_{s\in\mathcal{S}}, q^+:\mathcal{SA}\times\mathcal{S}\to\RR_{\geq0}, \bm{c}:\mathcal{SA}\to\RR^k  \right\rangle,$$
where the state space, action space, and cost functions are all defined as before, but we now define the \emph{generalised transition rates} $q^+(s,a,s')$, satisfying $0\leq q^+(s,a,s) < +\infty$, as the positive transition rate from state $s$ to state $s'$ when taking action $a$. This allows the explicit representation of transition rates from a state back to itself. From the generalised transition rates, we can define the \emph{negative generalised transition rates} $q^-(s,a,s') := -\sum_{s''\in\mathcal{S}}q^+(s,a,s'')\II\{s=s'\}$ as the negative outbound flows of the transition rates. We can then recover the standard transition rates $q(s,a,s'):=q^+(s,a,s') + q^-(s,a,s')$, where the rates in standard form satisfy:
\begin{align}
    +\infty > q(s,a,s') &\geq 0, \text{ for all }s'\neq s,\\
    -\infty < q(s,a,s') &\leq 0, \text{ for }s'=s,\\
    \sum_{s'\in\mathcal{S}}q(s,a,s') &= 0.
\end{align}
As such, all transitions from a state back to itself are ignored by any analysis that relies only on the $q$ function, and a standard CTMDP is recovered. It is also possible to go in the other direction by starting with a standard CTMDP with transition rates $q(s,a,s')$ and choosing $q^+(s,a,s')$ such that $q^+(s,a,s') = q(s,a,s')$ for any $s,a,s'$ such that $s\neq s'$, and $q^+(s,a,s)$ {\color{black}can then be chosen as any non-negative value, as any such choice yields the same original $q$ function when converting back}.

Upon first inspection, the notion of generalised transition rates may seem superfluous. However, we shall use it to define a range of useful equivalence classes over states and state-action pairs. We give the following remark to demonstrate what is meant by saying that two states are equivalent. 

\begin{remark}
Consider the following two-state continuous-time Markov chain in standard form:
$$Q=\begin{bmatrix}
    -1 & 1\\
    1 & -1
\end{bmatrix}.$$
This transition rate matrix can be obtained from the following positive transition rate matrix and its associated negative transition rate matrix:
$$Q^+ = \begin{bmatrix}
    1 & 1\\
    1 & 1
\end{bmatrix},\\
Q^- = \begin{bmatrix}
    -2 & 0\\
    0 & -2
\end{bmatrix}.$$
Here, it is easy to see that the two states are in some way ``equivalent", in that they transition back to themselves and to each other at the same rate, but this information gets lost in the standard form. This notion of equivalence will be rigorously developed in \autoref{paper2:model:ctmdp:theory}.
\end{remark}

For any unichain CTMDP, we consider the infinite-horizon long-run average multi-objective LP formulation following \cite{xianping2009continuous}, which we call the Original Problem (OP):
\begin{align}
    &\text{(OP)} &\min_{\bm{\pi}}   &\,g_1 = \sum_{(s,a)\in\mathcal{SA}}c_1(s,a)\pi(s,a) \label{op-obj1}\\
    &            & \vdots \nonumber\\
    &            &\min_\pi     &\,g_K = \sum_{(s,a)\in\mathcal{SA}}c_K(s,a)\pi(s,a) \label{op-objK}\\
    &            & \text{s.t.} &\sum_{(s,a)\in\mathcal{SA}}q(s,a,s')\pi(s,a) = 0,\text{ for all }s'\in\mathcal{S}, \label{op-cons-balance}\\
    &            &             &\sum_{(s,a)\in\mathcal{SA}}\pi(s,a) = 1, \label{op-cons-norm}\\
    &            &             &\pi(s,a)\geq 0,\text{ for all }(s,a)\in\mathcal{SA}. \label{op-cons-pos}
\end{align}
Under a mixed-objectives scalarisation, this formulation is also valid for weakly-commu\-nicating CTMDPs, as such CTMDPs always permit a unichain optimal policy \citep[Chapter 9]{puterman2014markov}. Lines \eqref{op-obj1} to \eqref{op-objK} give the $K$ objectives of the problem. {\color{black} Each decision variable $\pi(s,a)$ is a state-action frequency, i.e. the long-run probability of being in state $s$ and taking action $a$; these variables extend the notion of a stationary distribution in standard Markov chain analysis to the MDP setting.} Constraints \eqref{op-cons-balance} give the CTMDP balance equations which enforce the stochastic dynamics of the problem, and constraints \eqref{op-cons-norm}-\eqref{op-cons-pos} ensure that $\pi$ is a valid probability distribution. For $K>1$, this is a multi-objective problem for which we need to consider the Pareto front. A Pareto-optimal solution may, in general, be a \emph{randomized policy}, where for any recurrent state $s$ the probability of taking action $a$ in state $s$ is given by $\mu(s,a) = \pi(s,a)/\sum_{a'\in\mathcal{A}(s)}\pi(s,a')$ \citep{roijers2013momdp_survey}. These policies correspond to non-basic feasible solutions of OP. Despite this, there always exists a subset of the Pareto-optimal policies which are deterministic (i.e., basic), and our analysis focuses on such policies. This is because these policies are likely to be more intuitive, easier to implement, and more explainable to the decision-maker than randomized policies. As such, we argue that the goal of finding all DS policies that are Pareto-optimal with respect to the set of stochastic stationary policies is a reasonable one, as such a set still offers great flexibility to the decision maker. The rest of this subsection is dedicated to the development of a theory of \emph{equivalence} over the states and state-action pairs of CTMDPs with self-transitions. This will be useful later for the analysis of impulsive CTMDPs, which can have many equivalent states or state-action pairs.

\subsubsection{Theory}
\label{paper2:model:ctmdp:theory}
We first define what it means for states or state-action pairs to be equivalent. To be clear, we are \emph{not} defining the well-known equivalence relations that give the communicating classes commonly studied for Markov chains, but a different notion of equivalence that is needed for some of the analysis in this paper. 
\begin{definition}[Equivalence]
    For any finite-state finite-action CTMDP with self-transitions, with vector costs $\bm{c}(s,a)$ and generalised transition rates $q^+$, we say that two state-action pairs $(s,a)$ and $(s',a')$ are:
    \begin{enumerate}[(i)]
        \item \emph{stochastically equivalent} if  $q^+(s,a,s'') = q^+(s',a',s'')$ for all states $s''$, and write $(s,a)\sim^{q^+} (s',a')$,
        \item \emph{cost equivalent} if $\bm{c}(s,a)=\bm{c}(s',a')$, and write $(s,a)\sim^{\bm{c}}(s',a')$,
        \item \emph{totally equivalent} if $(s,a)\sim^{q^+} (s',a')$ and $(s,a)\sim^c(s',a')$, and write $(s,a)\sim (s',a')$.
    \end{enumerate}
    Additionally, for any DS policy $\mu$, we say that two states $s,s'$ are:
    \begin{enumerate}[(i)]
        \item \emph{stochastically equivalent} if $(s,\mu(s))\sim^{q^+}(s',\mu(s'))$, and write $s\sim_\mu^{q^+}s'$,
        \item \emph{cost equivalent} if $(s,\mu(s))\sim^{\bm{c}}(s',\mu(s'))$, and write $s\sim_\mu^{\bm{c}}s'$,
        \item \emph{totally equivalent} if $(s,\mu(s))\sim(s',\mu(s'))$, and write $s\sim_\mu s'$.
    \end{enumerate}
    Finally, for any equivalence relation $R$ over the feasible state-action space $\mathcal{SA}$ and any DS policy $\mu$, we can define the \emph{$\mu$-adapted equivalence relation} $R_\mu$ over the state space $\mathcal{S}$ as $s\,R_\mu\,s' \iff (s,\mu(s))\,R\,(s',\mu(s'))$
\end{definition}
As these are all based on equalities, it is immediately evident that the binary relations $\sim^{q^+},\sim^{\bm{c}}, \sim$, and their $\mu$-adapted counterparts are all equivalence relations. We can therefore define the equivalence classes $\widehat{s,a}$ under any of these relations and the partitions $\widehat{\mathcal{SA}}$ of $\mathcal{SA}$ induced by those equivalence relations. For any policy $\mu$, we can also define the equivalence classes $\widehat{s}$ under any of the $\mu$-adapted relations and the partitions $\widehat{\mathcal{S}}$ of $\mathcal{S}$ induced by those equivalence relations.

Now, under any DS policy $\mu$, we wish to construct and analyse the stochastic process formed by mapping the state of the chain to its stochastic equivalence class. In fact, more generally than this, we can construct and analyse the continuous-time Markov chain (CTMC) formed from any equivalence relation $R$ that is \emph{as strong as} stochastic equivalence, i.e., $s\,R\,s' \implies s\sim_\mu^{q^+}s'$. 
{\color{black} Following \cite{kemenySnell1976FiniteMarkovChains}, we say that a controlled (discrete-time) Markov chain is \emph{strongly lumpable} with respect to the partition of $\mathcal{S}$ induced by $R$, where a Markov chain over state space $\mathcal{S}$ is said to be strongly lumpable with respect to a partition $\widehat{\mathcal{S}}$ of $\mathcal{S}$ if:
$$\sum_{s\in\widehat{s}}p(s',s)=\sum_{s\in\widehat{s}}p(s'',s),\text{ for all } s',s''\in\widehat{s}',\text{ and for all }\widehat{s},\widehat{s}'\in\widehat{\mathcal{S}}.$$
This definition extends naturally to the continuous-time case via uniformisation. As such, the following \autoref{prop:classesAreChains}, \autoref{coro:classesAreChainsCont}, and \autoref{prop:aggregatedCTMChasExpectedDistro} all immediately follow; however, we provide direct alternative proofs in the {\color{black} supplementary material.}
We define $q_\mu^+(s,s') := q^+(s,\mu(s),s')$ and recognise that this value depends only on the class of $s$ with respect to $R$, and therefore we can write $q_\mu^+(\widehat{s},s')$ without loss of generality. We then define $q_\mu^+(\widehat{s},\widehat{s}'):=\sum_{s''\in\widehat{s}'}q_\mu^+(\widehat{s},s'')$, which intuitively means that we expect the generalised transition rate from class $\widehat{s}$ to $\widehat{s}'$ to be the cumulative generalised transition rate across all members of $\widehat{s}$ to any member of $\widehat{s}'$. We therefore also define:
\begin{align*}
    q_\mu^-(\widehat{s}) &:= -\sum_{\widehat{s}'\in\widehat{\mathcal{S}}}q_\mu^+(\widehat{s},\widehat{s}')\\
    q_\mu^-(\widehat{s},\widehat{s}') &:= \II\{\widehat{s}=\widehat{s}'\}q_\mu^-(\widehat{s}),
\end{align*}
to give the negative generalised transition rates for any $\widehat{s},\widehat{s}'\in\widehat{\mathcal{S}}$. We also define $q_\mu(\widehat{s},\widehat{s}'):=q_\mu^+(\widehat{s},\widehat{s}') + q_\mu^-(\widehat{s},\widehat{s}')$ to get the transition rates in standard form. Now, let $S^\mu(t)$ be the CTMC induced by a CTMDP with self-transitions controlled by a DS policy $\mu$, and let $\lcurl S_k^\mu\rcurl_{k=1}^\infty$ be $S^\mu(t)$ uniformised using some uniformisation parameter $\Delta$, following \cite{uniformisationSerfozo1979}. We want to investigate the behaviour of the equivalence classes of this process with respect to any equivalence relation $R$ over $\mathcal{S}$ which is as strong as $\sim_\mu^{q^+}$, i.e., $s Rs'\implies s\sim_\mu^{q^+}s'$. {\color{black} In other words, we wish to investigate the process that has been lumped with respect to the partition induced by such $R$.} Define $\widehat{S_{k}^\mu} := \varphi_R(S_k^\mu)$, where $\varphi_R$ is the \emph{canonical surjection} of $R$ (i.e., the function that maps an element to its equivalence class under $R$), so that $\widehat{S_k^\mu}$ is the stochastic process of the equivalence classes of the controlled CTMDP.

\begin{proposition}
\label{prop:classesAreChains}
   $\left\{\widehat{S^\mu_{k}}\right\}_{k=1}^\infty$ is a Markov chain over $\widehat{\mathcal{S}}$ with transition probabilities:
    $$p_\Delta(\widehat{s},\widehat{s}') = \begin{cases}
         1 + q(\widehat{s},\widehat{s})\Delta,& \widehat{s}' = \widehat{s}, \\
         q(\widehat{s},\widehat{s}')\Delta, & \text{otherwise.} 
    \end{cases}$$   
\end{proposition}
\begin{proof}
Follows immediately from lumpability with respect to $R$. See Supplementary Material 1.1 for a direct proof.
\end{proof}
We can also prove a similar equivalence property in continuous time. Under an equivalence relation $R$ as previously specified, let $\widehat{S^\mu}:=\varphi_R\circ S^\mu$, be the stochastic process of the equivalence classes of $S^\mu$, where we recall $S^\mu$ is the CTMC induced by a CTMDP with self-transitions controlled by a DS policy $\mu$.
\begin{corollary}
\label{coro:classesAreChainsCont}
The stochastic process $\widehat{S^\mu}(t)$ is a continuous-time Markov chain over $\widehat{\mathcal{S}}$, equivalent to $\{\widehat{S^\mu_{k}}\}_{k=1}^\infty$ under uniformisation, with transition rates given by $q(\widehat{s},\widehat{s}')$ as previously defined.
\end{corollary}
\begin{proof}
    Follows immediately from lumpability with respect to $R$. See Supplementary Material 1.2 for a direct proof.
\end{proof}

We refer to the CTMC $\widehat{S^\mu}$ as the \emph{class-aggregated} CTMC. We now wish to find the relationship between any stationary distribution of the original controlled CTMDP, $\bm{\pi}^\mu$, and a stationary distribution of our class-aggregated CTMC.

\begin{proposition}\label{prop:aggregatedCTMChasExpectedDistro}
If $\bm{\pi}^\mu$ is a stationary distribution for a CTMC, then a class-aggregated CTMC $\widehat{S}^\mu$ under a suitable equivalence relation R has stationary distribution $\widehat{\bm{\pi}^\mu}$ satisfying $\widehat{\pi^\mu}(\widehat{s}) = \sum_{s\in\widehat{s}}\pi^\mu(s)$.
\end{proposition}
\begin{proof}
Follows immediately from lumpability with respect to $R$. See Supplementary Material 1.3 for a direct proof.   
\end{proof}
This proposition can naturally be extended from the setting of controlled CTMDPs to the solutions $\pi(s,a)$ of OP that correspond to DS policies. Let $\mu$ be such a policy and let $R$ be an equivalence relation as strong as $\sim^{q^+}$ with equivalence classes $\widehat{s,a}$, and let $R_\mu$ be its $\mu$-adapted equivalence relation with classes $\widehat{s}$. We define the \emph{$R$-aggregated state-action frequencies} as $\widehat{\pi}(\widehat{s,a}) := \sum_{(s,a)\in\widehat{s,a}}\pi(s,a)$.

\begin{lemma}\label{lemma:s-aHat}
    Let $\pi(s,a)$ be the basic feasible solution to the original problem \emph{OP} corresponding to $\mu$, so $\pi(s,\mu(s)) = \pi^\mu(s)$ and $\pi(s,a) = 0$ for $a\neq\mu(s)$. The $R$-aggregated state-action frequencies satisfy the following equality:
    $$\widehat{\pi}(\widehat{s,a}) =  \begin{cases}
         \widehat{\pi^\mu}(\widehat{s}), & \text{if there exists $(s',a')\in\widehat{s,a}$ such that $a'=\mu(s')$}, \\
         0, & \text{otherwise.} 
    \end{cases}$$
\end{lemma}
\begin{proof}
See \autoref{app:lemma:s-aHat}.
\end{proof} 

In the same way that generalised transition rates only depend on stochastic equivalence classes, cost rates $\bm{c}(s)$ depend only on the cost equivalence classes of $s$, so we can write $\bm{c}(\widehat{s,a}) := \bm{c}(s,a)$ for any $(s,a)\in\widehat{s,a}$, and for a fixed policy $\mu$ we can write $\bm{c}^\mu(\widehat{s}) := \bm{c}^\mu(s)$ for any $s\in\widehat{s}$. We can therefore associate these costs with the Markov chain $\widehat{S^\mu}$ and obtain a corollary to \autoref{prop:aggregatedCTMChasExpectedDistro}. 
\begin{corollary}
    Consider a CTMDP with self-transitions $\mathfrak{C}$. Let $\widehat{\mathcal{S}}$ be the partition of the state space $\mathcal{S}$ with respect to an equivalence relation $R$ that is as strong as $\sim_\mu$, i.e., $\mu$-adapted total equivalence. The class-aggregated CTMC has the same long-run average cost as the original controlled CTMDP.
\end{corollary}
\begin{proof}
    We see that: $$\sum_{s\in\mathcal{S}}\bm{c}^\mu(s)\pi^\mu(s) = \sum_{\widehat{s}\in\widehat{S}}\sum_{s\in\widehat{s}}\bm{c}^\mu(\widehat{s})\pi^\mu(s)\\
        =\sum_{\widehat{s}\in\widehat{S}}\bm{c}(\widehat{s})\widehat{\pi^\mu}(\widehat{s}),$$
    so both the controlled CTMDP and the new CTMC with costs have the same long-run average costs.
\end{proof}
Intuitively, we have found that for any fixed DS policy, we can aggregate together any states with the same stochastic dynamics and costs, and the resulting aggregated CTMC behaves as desired. \autoref{paper2:model:impulsive} now follows with a formulation for CTMDPs with impulsive control.
\subsection{Impulsive CTMDPs}
\label{paper2:model:impulsive}
\subsubsection{Model}
\label{paper2:model:impulsive:model}
We wish to model \emph{impulsive control}, meaning that actions induce instantaneous transitions, and many such actions can be taken back-to-back in the same instant. This requires the development of some additional structure and assumptions on the CTMDP model:
\begin{enumerate}
    \item We assume that {\color{black}each state $s$ has some zero action} $0\in\mathcal{A}(s)$, and this is the only action that allows time to advance, i.e., it is the one gradual action for that state, using the terminology of \cite{dufour2015impulsive}. All other actions will be impulsive.
    \item We require all state-action pairs $(s,a)$ with non-zero $a$ to ``emulate'' a unique ``post-decision'' state denoted by $s\oplus a\in\mathcal{S}$ under the zero action for that state, where $\oplus:\mathcal{SA}\to\mathcal{S}$ is a binary operation satisfying $s\oplus0=s$ for all $s$. By emulate, we mean the following relations are true:
    \begin{align}
    q^+(s,a,s') &= q^+(s\oplus a, 0, s'), \label{impulsiveRates}\\
    \bm{c}(s,a) &= \bm{c}(s\oplus a, 0). \label{impulsiveCosts}
    \end{align}
    \item For any state $s$ we have $a_1\in\mathcal{A}(s)$ and some $a_2\in\mathcal{A}(s\oplus a_1)$, then there exists a unique action which we denote $a_1 \cdot a_2$, and this action must satisfy $s\oplus (a_1\cdot a_2) = (s\oplus a_1)\oplus a_2$.
\end{enumerate}

Assumption (1) ensures that there is only one way to allow ``time to pass'' in any given state, i.e., to stay in that state, as we are assuming that all possible controls or interventions are strictly impulsive. Generalisation to multiple gradual actions may be possible, but this is beyond the scope of this work. Assumption (2) emulates one-step impulsive control by ensuring that any state-action pair has the same stochastic dynamics and costs as some ``target'' state under that state's gradual action. As such, the process does not \emph{truly} instantaneously transition upon taking an impulsive action, but otherwise acts like it has. {\color{black}This notion of a post-decision state can be thought of as the continuous-time equivalent of the post-decision states found in \cite{powell2011adp}, which for some problems allows us to consider the ``state'' of the system after an action has been taken, but before the next random transition has been observed.} Additionally, as a result of the relations over $q$ and $\bm{c}$, it is enough to specify transition rates under the zero action $q^+(s,s') := q^+(s,0,s')$ and cost rates under the zero action $\bm{c}(s) := \bm{c}(s,0)$. {\color{black}Assumption (3) then ensures that the sequencing of a finite number of impulsive actions in a single instant, as allowed by other models of impulsive control, can always be emulated by a single ``sequenced'' action.} Again, the process does not truly jump between multiple states during that instant, but it otherwise acts as though it has landed in that final state. Without confusion, we can simplify the notion of Assumption (3) by writing $s\oplus a_1 \oplus a_2$ instead of $(s\oplus a_1)\oplus a_2$ or $s\oplus(a_1\cdot a_2)$ whenever it is clear that $s$ is a state and $a_1$ and $a_2$ are actions. Additionally, we may also write $a_1\oplus a_2$ instead of $a_1\cdot a_2$.  We additionally assume \emph{associativity} of $\oplus$ over actions, that is $(a_1\oplus a_2)\oplus a_3 = a_1\oplus(a_2\oplus a_3)$ always holds, so we can write $a_1 \oplus a_2 \oplus a_3$. 

{\color{black}
\begin{remark}
    If $\mathcal{A}(s)=\mathcal{A}$ for all $s\in\mathcal{S}$, we can say that $(\mathcal{A,\oplus})$ is a \emph{monoid} (i.e., group without inverses) that (right-)acts on $\mathcal{S}$ such that $q^+$ and $\bm{c}$ obey the monoid action in the expected way.
\end{remark}
}
Putting this all together, an impulsive CTMDP is a tuple:
$$\mathfrak{C} = \langle\mathcal{S},\left(\mathcal{A}(s)\right)_{s\in\mathcal{S}},q^+:\mathcal{S}\times\mathcal{S}\to\RR_{\geq0}, \bm{c}:\mathcal{S}\to\RR^k, \oplus:\mathcal{SA}\to\mathcal{S}\rangle,$$
from which a CTMDP with self-transitions can be recovered by ensuring $q^+(s,a,s')$ and $\bm{c}(s,a)$ follow the relations \eqref{impulsiveRates} and \eqref{impulsiveCosts}, respectively.

We now consider what an optimal policy for such a process might look like. If under an optimal policy $\mu$ we take action $a\neq 0$ under state $s$, then it seems intuitive that action $0$ should be taken in state $s \oplus a$, regardless of whether this state is reached via an impulsive action or a gradual action. On the other hand, if an optimal policy takes an impulsive action $a'$ in state $s \oplus a$, then it makes sense that it should take action $a\oplus a'$ in state $s$. In other words, we expect optimal policies to choose actions that result in instantaneous transitions to the most desirable states (those in which we allow time to pass), rather than choosing actions leading to states that they wish to exit from immediately. We formalise this notion by defining the property of \emph{coherence}.
\begin{definition}[Coherent Policy]
    Let $\mu$ be a DS policy for an impulsive CTMDP. We say $\mu$ is a \emph{coherent policy} if, for every state $s$ such that $\mu(s)=a\neq0$, we have $\mu(s\oplus a) = 0$.
\end{definition}
We will later prove that there always exists a coherent optimal policy, and that this policy can easily be constructed from a non-coherent DS optimal policy (if one exists). This will be useful because a generic solution algorithm for CTMDPs (linear programming, value iteration, etc.) may not always return a coherent policy, but to prove this we need to understand the behaviour of non-coherent policies. In general, it is possible for non-coherent policies not only to have $a_2=\mu(s\oplus a_1)\neq 0$, but also to have $a_3=\mu(s\oplus a_1\oplus a_2)\neq 0$ and so on, forming a ``chain'' of non-zero impulsive actions. 
The next definition formalises this idea.
\begin{definition}[Chain]
    Let $\mu$ be a DS policy for an impulsive CTMDP. Let $\{s_n\}_{n=1}^\infty$ be the sequence of states defined by $s_1=s$ and $s_{n} = s_{n-1}\oplus\mu(s_{n-1})$ for $n\geq 2$. Let $\{a_n\}_{n=1}^\infty$ be the sequence of actions defined by $a_n=\mu(s_n)$. We call $\{a_n\}_{n=1}^\infty$ the \emph{chain of $s$ under $\mu$}, and denote it by $\text{Chain}^\mu(s)$. We call $\{s_n\}_{n=1}^\infty$ the \emph{state chain of $s$ under $\mu$}. We also define the \emph{degree} of a chain by the number of non-zero actions in the chain.
\end{definition}
Using the notion of a chain, we can equivalently define coherent policies as DS policies in which the degree of the chain for every state is at most one. A chain can take on only two possible forms: a finite number of non-zero actions followed by an infinite sequence of zeros (i.e., $\{a_1,\dots,a_N,0,0,0,\dots\}$); or an infinite sequence of non-zero actions. This is due to the fact that if $\mu(s_n) = 0$ for some $n\geq 1$, then $s_{n+1}=s_n$ and hence $\mu(s_m) = \mu(s_n)$ for all $m\geq n$. The infinite sequence will always eventually enter into a loop, as the state space is finite and the sequence is infinite and deterministic, so it must eventually loop back on itself. If the state chain of $s$ repeats $s$ for index $n>1$, we call $s$ a \emph{looping state}. From this understanding of chains, we may suggest a method for constructing a coherent optimal policy from a non-coherent optimal policy, which we state as a theorem.
\begin{theorem}
\label{claim:coherentFromNonCoherent}
Let $\mu$ be a non-coherent optimal policy for an impulsive CTMDP, implying that there exists at least one state $s$ such that $\mu(s)\neq 0$ and $\mu(s\oplus\mu(s))\neq0$. A coherent optimal policy can be constructed following this two-step process:
\begin{enumerate}
    \item Set $\mu'(s) = \begin{cases}
         \mu(s), & s \text{ is not looping,}  \\
         0, & s \text{ is looping.}
    \end{cases}.$ \\
    The policy $\mu'$ is optimal.
    \item Set $\mu''(s) = a_1 \oplus \dots \oplus a_{N(s)}$, where $\{a_n\}_{n=1}^\infty$ is the chain of $s$ under $\mu'$ and $N(s)$ is the degree of that chain. This policy $\mu''$ is optimal and coherent.
\end{enumerate}
\end{theorem}
In short, we claim that coherent optimal policies can be constructed from non-coherent optimal policies by eliminating loops and aggregating chains of actions together. We defer the proof of \autoref{claim:coherentFromNonCoherent} to \autoref{paper2:model:impulsive:theory}. However, we \emph{do not} suggest that this extends well to suboptimal policies that result from the use of approximate methods such as reinforcement learning or bespoke heuristics. 
\begin{remark}
\label{claim:doNotSequenceSubOpt}
Let $a_1$ be a suboptimal action for $s$ and let $a_2$ be a suboptimal action for $s\oplus a_1$. The bound on suboptimality is multiplicatively worse for the sequenced action $a_1\oplus a_2$. In general, the bound on the error grows exponentially in the degree of the sequencing. 
\end{remark}
A more rigorous statement and proof of \autoref{claim:doNotSequenceSubOpt} is given in \autoref{app:boundExample}. As such, one should proceed with caution when considering the use of sequencing on suboptimal non-coherent policies obtained from approximate methods (reinforcement learning, heuristics, etc.). 
The remainder of this subsection provides the necessary mathematical treatment of \autoref{claim:coherentFromNonCoherent}, alongside other analysis of impulsive CTMDPs. 

\subsubsection{Theory}
\label{paper2:model:impulsive:theory}
The following lemma concerns the relative action value functions $Q^\mu$.
\begin{lemma}
    Consider an impulsive CTMDP. For any policy $\mu$ with constant gain $g^\mu$ and any state-action pair $(s,a)$, we have $Q^\mu(s,a) = Q^\mu(s\oplus a, 0)$.
\end{lemma}
\begin{proof}
    \begin{align*}
        Q^\mu(s,a) &= (c(s,a)-g^\mu)t^+(s,a) + \EE_{s'\sim q^+(s,a,\cdot)} h^\mu(s'), \text{ (by definition of $Q^\mu$)}\\
        &= (c(s\oplus a,0) - g^\mu)t^+(s\oplus a,0) + \EE_{s'\sim q^+(s\oplus a,0,\cdot)} h^\mu(s'),\text{ (by impulsivity)}\\
        &= Q^\mu(s\oplus a,0),
    \end{align*}    
    where $t^+(s,a) = 1/q^-(s,a,s)$ is the sojourn time of state $s$ under action $a$ when allowing for ``dummy'' transitions back to $s$, and $h^\mu$ is the bias (i.e., relative value function) of policy $\mu$.
\end{proof}
We wish to prove that, in the single objective case, there always exists a coherent policy that is optimal, and that this can be constructed by following \autoref{claim:coherentFromNonCoherent}. This then extends naturally to the multi-objective case, because for any \emph{supported} Pareto-optimal DS policy $\mu$, there exists an objective weighting $\bm{w}$ such that $\mu$ is optimal for single-objective cost function $c_{\bm{w}}(s)=\bm{w}\cdot\bm{c}(s)$, so there will then exist a coherent policy $\mu'$ that is also optimal for this weighting, implying that $\mu'$ is Pareto-optimal. The approach will require the use of \emph{sequencing operators}, which we define next. 
\begin{definition}[Sequencing Operators]
    Let $\mathcal{M}$ be the space of DS policies for an impulsive CTMDP. We call a function $\text{Seq}_k:\mathcal{M}\to\mathcal{M}$ the \emph{$k$th order sequencing operator} where $\text{Seq}_k\{\mu\}(s) = a_1\oplus\dots\oplus a_{k+1}$, and $\{a_n\}_{n=1}^\infty$ is the chain of $s$ under $\mu$. We call $\text{Seq}_k\{\mu\}$ the $k$th order sequenced policy.
\end{definition}
An equivalent and recursive definition is $\text{Seq}_1\{\mu\}(s) = \mu(s)\oplus\mu(s\oplus\mu(s))$, and $\text{Seq}_k\{\mu\}(s)=\text{Seq}_{k-1}\{\mu\}(s)\oplus\mu(s\oplus\text{Seq}_{k-1}\{\mu\}(s))$ for $k\geq 2$. This method of sequencing can be used to break loops and reduce chains to length 0 or 1, while maintaining optimality. To prove this, we must first prove that one step of sequencing maintains optimality.
\begin{lemma}
\label{lemma:OneStepSequencingIsOptimal}
    Let $Q^*$ be the optimal $Q$ value function for a single-objective impulsive CTMDP. If $a_1\neq0$ is optimal for some state $s$ and $a_2\neq0$ is optimal for $s\oplus a_1$, then $a_1\oplus a_2$ is also optimal for $s$.
\end{lemma}
\begin{proof}
    Suppose that there exists a state $s$ such that $s,a_1,a_2$ are as described. By optimality, $Q^*(s,a_1)\leq Q^*(s,a)$ for all states $s$ and actions $a$. For any such $s$, we have $Q^*(s\oplus a_1,0) = Q^*(s,a_1) \leq Q^*(s, a_1\oplus a_2) = Q^*(s\oplus a_1 \oplus a_2,0)$. However, because $a_2$ is optimal for $s\oplus a_1$, we also have $Q^*(s\oplus a_1 \oplus a_2,0) = Q^*(s\oplus a_1, a_2) \leq Q^*(s\oplus a_1, 0)$. Therefore $Q^*(s\oplus a_1,0) = Q^*(s\oplus a_1 \oplus a_2,0)$ and hence $Q^*(s,a_1) = Q^*(s,a_1 \oplus a_2)$, so action $a_1\oplus a_2$ is also optimal for $s$.
\end{proof}
We can use \autoref{lemma:OneStepSequencingIsOptimal} to show that for any optimal policy $\mu$, its sequenced forms are also optimal.
\begin{proposition}
\label{prop:sequencedOptimalPoliciesAreOptimal}    
Let $\mu$ be an optimal policy. Then $\text{Seq}_k\{\mu\}$ is an optimal policy for all $k$.
\end{proposition}
\begin{proof}
    We proceed by induction. It is immediate from \autoref{lemma:OneStepSequencingIsOptimal} that $\text{Seq}_1\{\mu\}$ is optimal. Now assume $\text{Seq}_k\{\mu\}$ is optimal for some $k$. It follows that for any state $s$, $\text{Seq}_k\{\mu\}(s)$ is optimal for $s$ and $\mu(s\oplus\text{Seq}_k\{\mu\}(s))$ is optimal for $s\oplus\text{Seq}_k\{\mu\}(s)$. Hence, by \autoref{lemma:OneStepSequencingIsOptimal}, $\text{Seq}_k\{\mu\}(s)\oplus\mu(s\oplus\text{Seq}_k\{\mu\}(s))$ is also optimal for $s$, so by definition $\text{Seq}_{k+1}\{\mu\}(s)$ is optimal for $s$. Therefore, $\text{Seq}_{k+1}\{\mu\}$ is also an optimal policy.
\end{proof}
We can use \autoref{prop:sequencedOptimalPoliciesAreOptimal} to show how to break loops and aggregate the actions in chains to obtain a coherent optimal policy.
\begin{proposition}
\label{prop:loopBreakingIsOptimal}
    Let $\mu$ be an optimal policy for an impulsive CTMDP, and let $s$ be a looping state with respect to $\mu$. Then action $0$ is optimal for $s$.
\end{proposition}
\begin{proof}
If $s$ is a looping state, then it reappears at some index $n>1$ in its state chain, so $s\oplus a_1\oplus\dots\oplus a_{n-1} = s$ and therefore $a_1\oplus\dots\oplus a_{n-1} = 0$. Action $a_1\oplus\dots\oplus a_{n-1}$ is also optimal for $s$ because $\text{Seq}_{n-2}\{\mu\}$ is an optimal policy. Therefore 0 is optimal for $s$.    
\end{proof}
As such, for any optimal policy $\mu$, the policy $\mu'$ described in \autoref{claim:coherentFromNonCoherent} is also optimal. The second part of the claim is simply a corollary of \autoref{prop:sequencedOptimalPoliciesAreOptimal}.
\begin{corollary}
    \label{coro:aggregatedChainsAreOptimal}
    Let $\mu$ be an optimal policy with no looping states, so all chains have a finite number of non-zero actions. Then, for any state $s$, $a_1\oplus\dots\oplus a_{N(s)}$ is optimal for state $s$, where the $a_n$ form the chain of $s$ with degree $N(s)$.
\end{corollary}
\autoref{prop:loopBreakingIsOptimal} and \autoref{coro:aggregatedChainsAreOptimal} together provide a proof of \autoref{claim:coherentFromNonCoherent}. This confirms that we can sequence together actions from a non-coherent optimal policy to obtain a coherent optimal policy. As previously mentioned, this naturally extends to supported Pareto-optimality (i.e., mixed-objective optimality) in the multi-objective case. 


So far, we have shown that transition rates, costs, value functions, and policies behave as expected when impulsive actions are implemented in the suggested manner. However, extra care is needed when examining the long-run probability of being in a particular state. Consider the CTMC induced by an impulsive CTMDP controlled by a coherent policy $\mu$, denoted by $S^\mu(t)$. Let $\bm{\pi}^\mu$ be the stationary distribution, and suppose there exists a recurrent state $s$ such that $\mu(s)=a\neq0$. As $s$ is recurrent, we have $\pi^\mu(s)>0$. However, in order to emulate impulsive actions, it would be better to have $\pi^\mu(s) = 0$, as the policy $\mu$ chooses to exit from state $s$ immediately. We can also think of this non-zero $\pi^\mu(s)$ value as detracting from the probability of being in state $s\oplus a$, the instantaneous post-decision state. To resolve this issue, we must build upon the equivalence relations introduced in the previous subsection and investigate the probabilities with respect to the equivalence relations. In this way, the probability of being in a state with an impulsive action can be reconciled with the probability of being in its target state in a way that makes sense. These equivalence relations are also useful in the analysis of decomposable impulsive CTMDPs in \autoref{paper2:model:impulsive}. We proceed to define a unique sense of equivalence over impulsive CTMDPs.
\begin{definition}[Impulsive Equivalence]
    For any impulsive CTMDP, we say two feasible state-action pairs $(s_1,a_1),(s_2,a_2)\in\mathcal{SA}$ are \emph{impulsively equivalent} if $s_1\oplus a_1 = s_2\oplus a_2$, and write $(s_1,a_1)\sim_\oplus(s_2,a_2)$. For an impulsive CTMDP controlled by a DS policy $\mu$, we say two states $s_1$ and $s_2$ are \emph{impulsively equivalent} if $(s_1,\mu(s_1))\sim_\oplus(s_2,\mu(s_2))$, and write $s_1\sim_{\oplus\mu} s_2.$
\end{definition}
An immediate consequence of this definition is that two impulsively equivalent states are also equivalent with respect to $\sim_\mu$. Impulsive equivalence is therefore a strictly stronger condition than total equivalence. We demonstrate this using the following example.
\begin{example}
    Consider a controlled impulsive CTMDP with state space $\mathcal{S} = \{1,2,3,4\}$ and action spaces:
    $$\mathcal{A}(1) = \mathcal{A}(2) = \{0,1\},\, \mathcal{A}(3) = \mathcal{A}(4) = {0}.$$
    Let $\mu(1) = \mu(2) = 1$, $1\oplus1 = 3$ and $2\oplus 1 = 4$. Suppose all costs are 0, so all states are trivially cost equivalent. Suppose $q^+(3,0,3) = q^+(3,0,4) = q^+(4,0,3) = q^+(4,0,4) = 1$. Then all states are stochastically equivalent under $\mu$ and hence are all equivalent under $\mu$. However, state-action pairs $(1,1)$ and $(2,1)$ are not impulsively equivalent because $1\oplus1\neq2\oplus1$.
\end{example}
We also have that $\bm{c}(s_1,a_1) = \bm{c}(s_2,a_2)$ for impulsively equivalent pairs $(s_1,a_1)$ and $(s_2,a_2)$, so we can write $\bm{c}(\widehat{s_1,a_1})$ without any ambiguity. We can immediately connect this sense of equivalence to the notion of coherence.
\begin{lemma}
    \label{lemma:coherenceImpliesRelated}
    Let $\mu$ be a coherent policy. Then $s\sim_{\oplus\mu}s\oplus\mu(s)$.
\end{lemma}
\begin{proof}
    We see that $(s,\mu(s))\sim_\oplus (s\oplus\mu(s),0) = (s\oplus\mu(s), \mu(s\oplus\mu(s)))$, so $s\sim_{\oplus\mu}s\oplus(\mu(s))$.
\end{proof}
This intuitively tells us that any state is equivalent to its post-decision state under a coherent policy. This allows us to aggregate all states with the same post-decision state (including the post-decision state itself) into one equivalence class, and as such we can report the values $\widehat{\pi^\mu}(\widehat{s})$ for this post-decision state as if the actions were truly instantaneous. Next, we show that this kind of post-decision state in a class is unique.
\begin{lemma}
    Let $\mu$ be a coherent policy, and let $\widehat{s}$ be an equivalence class with respect to $\sim_{\oplus\mu}$. There is exactly one $s\in\widehat{s}$ such that $\mu(s) = 0$, which we call the \emph{post-decision state} of the class.
\end{lemma}
\begin{proof}
We first prove existence of such a state. The class $\widehat{s}$ is a non-empty set and therefore contains some state $s'$. It follows that either $\mu(s')=0$, so such a state exists, or $\mu(s')\neq 0$ and therefore the class also contains $s'\oplus\mu(s')$ by \autoref{lemma:coherenceImpliesRelated}, which by definition of coherence satisfies $\mu(s'\oplus\mu(s')) = 0$. To show uniqueness, suppose there exist two states $s_1,s_2\in\widehat{s}$ such that $\mu(s_1)=\mu(s_2)=0$. Then since $s_1\sim_{\oplus\mu}s_2$ we have $(s_1,\mu(s_1))\sim_\oplus (s_2,\mu(s_2))$ and hence $(s_1,0)\sim_\oplus(s_2,0)$. Therefore $s_1\oplus0=s_2\oplus0$, implying that $s_1=s_2$. 
\end{proof}
We can interpret the value $\widehat{\pi^\mu}(\widehat{s})$ as the probability of being in the one state in $\widehat{s}$ in which we take the gradual action, rather than viewing this probability as being `distributed' between the states which emulate this state impulsively. To be more precise, by \autoref{coro:classesAreChainsCont}, we know that if $S^\mu$ is the CTMC induced by controlling an impulsive CTMDP, then the class-aggregated process $\widehat{S^\mu}$ is also a CTMC, and $\widehat{\pi^\mu}(\widehat{s})$ is the stationary distribution of that CTMC. We can then relabel the $\widehat{s}$ to their post-decision states to obtain a CTMC that always transitions between post-decision states, such that no time is ever spent in states where a non-zero action is taken. As such, the desired behaviour of instant transitions from impulsivity is satisfied up to an equivalence relation when the impulsive CTMDP is controlled by a coherent policy, as desired. This concludes our general results for impulsive CTMDPs, and \autoref{paper2:model:decomp} now follows with the introduction and analysis of decomposable impulsive CTMDPs.

\subsection{Decomposable Impulsive CTMDPs}
\label{paper2:model:decomp}
\subsubsection{Model}
\label{paper2:model:decomp:model}
We build towards the definition of a decomposable impulsive CTMDP in two stages. The first stage deals with the stochastic dynamics of the system whilst ignoring the costs, and the second stage deals with the structure of the cost function. In order to define a problem that can be broken into smaller problems, we take the approach of defining the smaller problems first and then combining these together. This motivates our first definition.
\begin{definition}[Cost-Free Direct Product of Impulsive CTMDPs]
    Let $\{\mathfrak{C}_i\}_{i=1}^I$ be a finite indexed set of impulsive CTMDPs  $\mathfrak{C}_i = \langle\mathcal{S}_i,\left(\mathcal{A}_i(s_i)\right)_{s_i\in\mathcal{S}_i},q^+_i:\mathcal{S}_i\times\mathcal{S}_i\to\RR_{\geq0}, \oplus_i:\mathcal{SA}_i\to\mathcal{S}_i\rangle.$ We define the \emph{cost-free direct product} of these impulsive CTMDPs as the costless impulsive CTMDP:
    $$\bigotimes_{i=1}^I\mathfrak{C}_i := \langle\mathcal{S},\left(\mathcal{A}(s)\right)_{s\in\mathcal{S}}, q^+:\mathcal{S}\times\mathcal{S}\to\RR_{\geq0}, \oplus:\mathcal{SA}\to\mathcal{S}\rangle,$$
    satisfying:
    \begin{align*}
        \mathcal{S} &= \bigtimes_{i=1}^I\mathcal{S}_i\\
        \mathcal{A}(\bm{s}) &= \bigtimes_{i=1}^I \mathcal{A}_i([\bm{s}]_i)\\
        q^+(\bm{s},\bm{s}') &= \begin{cases}
             q_i^+([\bm{s}]_i, s'_i), & \bm{s}' = ([\bm{s}]_1,\dots,s_i',\dots,[\bm{s}]_I) \text{ for some }i,  \\
             0, & \text{otherwise,} 
        \end{cases}.\\
        \bm{s} \oplus \bm{a} &= ([\bm{s}]_1\oplus_1 [\bm{a}]_1,\dots,[\bm{s}]_I\oplus_I[\bm{a}]_I).
    \end{align*}
    Additionally, we refer to $\mathfrak{C}_i$ as the \emph{$i$th compartment} of the product CTMDP. 
\end{definition}
By construction, this CTMDP is also impulsive, and the state space, action spaces and impulsive operator take the expected forms. The function $q^+$ requires more explanation: the structure of $q^+$ ensures that only one of the compartments can update upon a transition, whilst the others remain the same. This is because the rate of such a transition should depend only on the state (or ``sub-state'') of that compartment, and not on the rest of the composite state. 

\begin{remark}
{\color{black} A subtlety of this definition is that compartments are not entirely stochastically independent, even if ``controlled separately'' (which we will define shortly). Consider a two-compartment product impulsive CTMDP in state $\bm{s}=(s_1,s_2)$ under non-zero action $\bm{a}=(a_1,a_2)$. By definition, we have that $q^+(\bm{s},\bm{a},\bm{s}')=q^+(\bm{s}\oplus\bm{a},\bm{s}')$, so all possible next states are of the form $(s_1\oplus a_1,s_2')$ or $(s_1',s_2\oplus a_2)$. This means that an ``event'' in one compartment triggers all other compartments to simultaneously transition to ``remember'' the action that was previously being applied. Despite this apparent implicit dependence between compartment processes, we will later prove, under suitable conditions, that the controlled product processes still decompose into separate independent processes.} 
\end{remark}

Now, using this direct product, we can provide a further definition.
\begin{definition}[Stochastic Decomposability] An impulsive CTMDP $$\mathfrak{C} = \langle\mathcal{S},\left(\mathcal{A}(s)\right)_{s\in\mathcal{S}}, q^+:\mathcal{S}\times\mathcal{S}\to\RR_{\geq0}, \bm{c}:\mathcal{SA}\to\RR^k, \oplus:\mathcal{SA}\to\mathcal{S}\rangle$$is called \emph{stochastically decomposable with $I$ compartments} if the tuple $\langle\mathcal{S},\left(\mathcal{A}(s)\right)_{s\in\mathcal{S}}, q^+, \oplus\rangle$ can be expressed as a cost-free direct product of impulsive CTMDPs $\{\mathfrak{C}_i\}_{i=1}^I$. A choice of such sub-CTMDPs then gives the \emph{compartments} of the stochastically decomposable impulsive CTMDP.
\end{definition}
Trivially, any impulsive CTMDP is stochastically decomposable with one compartment. Also, the choice of compartmentalization is not unique. For example, the cost-free direct product of three compartments could be decomposed back into just two compartments, with two of the three original compartments remaining joined together. Before we turn our attention to the structure of the cost functions for a decomposable impulsive CTMDP, we must first consider a certain type of policy to be used for controlling such a process. Heuristically, we might think that it would make sense to control the compartments of the CTMDP separately, only making decisions for a compartment based on the sub-state and not the whole state. We argue that taking this approach could allow us to optimise the compartment CTMDPs separately, resulting in subproblems that are far easier to solve. We formalise this idea by introducing \emph{separable policies}.
\begin{definition}[Separable Policy]
Let $\mu$ be a DS policy for a stochastically decomposable impulsive CTMDP with compartments $\{\mathfrak{C}_i\}_{i=1}^I$. We say that $\mu$ is a \emph{separable policy} if there exist functions $\mu_i$ such that, for all $\bm{s}\in\mathcal{S}$, we can write $\mu(\bm{s}) = (\mu_1([\bm{s}]_1), \mu_2([\bm{s}]_2),\dots, \mu_I([\bm{s}]_I))$. That is, the action taken on compartment $i$ depends only on the state of compartment $i$.
\end{definition}
\begin{remark}
The flexibility of separable policies relies on the \emph{choice} of decomposition. As such, we advise caution when considering the choice of decomposition when limiting to such policies. It is not always necessarily best to decompose a problem as much as possible, as heuristics based on the use of separable policies may sometimes perform better when certain compartments are left non-decomposed. 
\end{remark}
Next, we discuss decomposability of the cost function. We do not wish to restrict attention to cost functions that are simply linear combinations of cost functions from the compartment CTMDPs. The following definition allows for a broader range of cost functions.
\begin{definition}[Decomposable Cost Function]
     Let $\mathfrak{C}$ be a stochastically decomposable impulsive CTMDP with $I$ compartments and $K$ objectives. Let $\bm{g}:[0,1]^{|\mathcal{SA}|}\to\RR^K, \bm{g}(\pi) = \sum_{s\in\mathcal{S},a\in\mathcal{A}(s)}\bm{c}(s\oplus a)\pi(s,a)$ be the function mapping every solution of OP to its long-run average cost. The cost function $c_k:\mathcal{S}\to\RR$ is a \emph{decomposable cost function} if there exists a {\color{black}component-wise strictly} increasing function $f_k:\RR^I\to\RR$ and functions $c_{i,k}:\mathcal{S}_i\to\RR$ for $i\in\{1,\dots,I\}$, which we call the \emph{$k$th sub-state costs}, such that for every $\bm{\pi}$ corresponding to a coherent separable policy $\mu$, we have:
    $$g_k(\pi) = f_k\left(\sum_{(s_1,a_1)\in\mathcal{SA}_1}c_{1,k}(s_1\oplus a_1)\pi_1(s_1,a_1),\dots,\sum_{(s_I,a_I)\in\mathcal{SA}_I}c_{I,k}(s_I\oplus a_I)\pi_I(s_I,a_I)\right),$$
    where $\pi_i(s_i,a_i) := \sum_{\bm{s},\bm{a}:[\bm{s},\bm{a}]_i = (s_i,a_i)}\pi(\bm{s},\bm{a})$ is the probability that compartment $i$ is in sub-state $s_i$ and undergoing sub-action $a_i$.
\end{definition}
\begin{remark}
    We may write the right-hand-side of the above expression in shorthand form as:
    $$f_k\left(\left(\sum_{(s_i,a_i)\in\mathcal{SA}_i}c_{i,k}(s_i\oplus a_i)\pi_i(s_i,a_i)\right)_{i=1}^I\right).$$
\end{remark}

The above definition is very general and quite verbose, so we discuss a few simple cases for the structure of a cost function and propose that each case implies decomposability.
\begin{definition}[Types of Cost Function]
     Let $\mathfrak{C}$ be a stochastically decomposable impulsive CTMDP with $I$ compartments and $K$ objectives. Then $c_k:\mathcal{S}\to\RR$ is:
     \begin{itemize}
         \item an \emph{additive cost function} if there exist functions $c_{i,k}:\mathcal{S}_i\to\RR$ such that $c_k(\bm{s}) = \sum_{i=1}^I c_{i,k}([\bm{s}]_i).$
        \item a \emph{multiplicative cost function} if there exist functions $c_{i,k}:\mathcal{S}_i\to\RR$ such that $c_k(\bm{s}) = \prod_{i=1}^I c_{i,k}([\bm{s}]_i).$
        \item  a \emph{logical-OR cost function} if there exist subsets $E_i\subset\mathcal{S}_i$ for each $i$ such that $c_k(\bm{s}) = \II\lcurl\bigvee_{i=1}^I [\bm{s}]_i\in E_i\rcurl.$
     \end{itemize}
     
\end{definition}
The additive cost function captures the simple idea of adding the costs of individual compartments, which often makes sense when the cost function represents monetary costs. However cost functions can sometimes have alternative interpretations. The logical-OR cost function is particularly useful if there is a set of states in each compartment that should be avoided, and therefore we aim to avoid any state of the whole system that contains such sub-states of compartments.
\begin{proposition}
\label{prop:addMultOR_are_decompable}
    Additive cost functions, multiplicative cost functions, and logical-OR cost functions are all decomposable cost functions.
\end{proposition}
\begin{proof}
    See \autoref{app:addMultLogicalORAreDecomp}
\end{proof}
Note that these cost functions are not exhaustive, and there may be other useful types of decomposable cost functions. Next, we define the decomposable impulsive CTMDP.
\begin{definition}[Decomposable Impulsive CTMDP]
    An impulsive CTMDP $\mathfrak{C}$ is called a \emph{decomposable impulsive CTMDP} if it is stochastically decomposable and each of its scalar cost functions $c_k$ is a decomposable cost function.
\end{definition}
This concludes the construction of the modelling framework. Those wishing to model a problem as a decomposable impulsive CTMDP may choose to construct their model directly and then prove that it is decomposable. Alternatively, they may choose to construct the compartment CTMDPs first and then use the direct product to construct the full model, and define cost functions over this post-hoc. The remainder of this subsection is dedicated to the analysis of stochastically decomposable impulsive CTMDPs.

\subsubsection{Theory}
\label{paper2:model:decomp:theory}
The following lemma links the concepts of separable policies and coherence.
\begin{lemma}
A separable policy $\mu$ is coherent if and only if each subpolicy $\mu_i$ is coherent. \label{lemma:coherentIffSubCoherent}
\end{lemma}
\begin{proof}
    Let $\mu$ be coherent and separable, so there exist policies $\mu_i:\mathcal{S}_i\to\mathcal{A}_i$ such that $\mu(\bm{s}) = (\mu_1([\bm{s}]_1),\dots,\mu_I([\bm{s}]_I))$ for all $i$. We also have that $\mu(\bm{s}\oplus\mu(\bm{s})) = 0$ for all $s$. Now:
    \begin{align*}
        \mu(\bm{s}\oplus\mu(\bm{s})) = \bm{0} \text{ for all }\bm{s}&\iff [\mu(\bm{s}\oplus\mu(\bm{s}))]_i = 0\text{ for all }\bm{s},i\\
        &\iff 
        \mu_i([\bm{s}\oplus\mu(\bm{s})]_i) = 0 \text{ for all }\bm{s},i\\
        &\iff \mu_i([\bm{s}]_i\oplus \mu_i([\bm{s}]_i)) = 0 \text{ for all }\bm{s},i\\
        &\iff \mu_i(v\oplus\mu_i(v)) = 0\text{ for all }i, v\in\mathcal{S}_i.
    \end{align*}
    So the bi-implication holds.
\end{proof}
Note that we now start to use $v$ for generic compartment states and $b$ for generic compartment actions, and we abbreviate policies that are coherent and separable to \emph{CS} policies. Now, for a choice of compartments $\{\mathfrak{C}_i\}_{i=1}^I$ for a stochastically decomposable impulsive CTMDP, we can form the impulsive equivalence relations over the state-action pairs of the compartment CTMDPs $\mathfrak{C}_i$, and we denote this relation by $\sim_{\oplus_i}$. This induces a partition over the compartment feasible state-action spaces $\mathcal{SA}_i$ which we denote $\widehat{\mathcal{SA}_i}$. Similarly, for any separable policy $\mu$, we can form impulsive equivalence relations over the states of the compartment CTMDPs using the sub-policies $\mu_i$, and denote these by $\sim_{\oplus\mu_i}$. In the following lemma and its corollary, we connect these two relations.
\begin{lemma}\label{lemma:state-action-equiv}
    Let $(\bm{s}_1,\bm{a}_1),(\bm{s}_2,\bm{a}_2)\in\mathcal{S}\mathcal{A}$ be two state-action pairs of a decomposable impulsive CTMDP. Then $([\bm{s}_1]_i,[\bm{a}_1]_i) \sim_{\oplus_i} ([\bm{s}_2]_i,[\bm{a}_2]_i)$ for all $i$ if and only if $(\bm{s}_1,\bm{a}_1) \sim_{\oplus} (\bm{s}_2,\bm{a}_2)$.
\end{lemma}
\begin{proof}
    We have $([\bm{s}_1]_i,[\bm{a}_1]_i) \sim_{\oplus_i} ([\bm{s}_2]_i,[\bm{a}_2]_i)$ for all $i$ if and only if $[\bm{s}_1]_i\oplus [\bm{a}_1]_i = [\bm{s}_2]_i\oplus [\bm{a}_2]_i$ for all $i$, which holds if and only if $\bm{s}_1\oplus \bm{a}_1 = \bm{s}_2\oplus \bm{a}_2$, which is identical to $(s_1,a_1)\sim_{\oplus}(s_2,a_2)$ by definition.
\end{proof}
\begin{corollary}\label{coro:state-equiv}
    For a separable policy $\mu$, let $\bm{s}_1,\bm{s}_2$ be two states of a decomposable impulsive CTMDP. Then $[\bm{s}_1]_i \sim_{\oplus\mu_i} [\bm{s}_2]_i$ for all $i$ if and only if $\bm{s}_1 \sim_{\oplus\mu} \bm{s}_2$.
\end{corollary}
\begin{proof}
    This is immediate from the definitions and \autoref{lemma:state-action-equiv}.
\end{proof}

Next, we show how a controlled decomposable impulsive CTMDP can be decomposed by making use of the equivalence relations developed. We first prove this for the uniformised discrete-time case, and then extend to the continuous-time case.
\begin{proposition}
\label{prop:decompositionWorks}
    Let $\lcurl S_k^\mu\rcurl_{k=1}^\infty$ be the discrete-time Markov chain obtained by uniformising (with parameter $\Delta)$ the CTMC induced by controlling a stochastically decomposable impulsive CTMDP with a coherent separable policy $\mu$. Let $S^\mu_{i,k} = [S^\mu_k]_i$ be the random variable determined by the $i$th component, i.e., the state of compartment $i$, for any $i$ at time $k$. Let $\widehat{S_{i,k}^\mu}:=\varphi_i\left(S_{i,k}^\mu\right)$ for all $i$, where $\varphi_i$ is the canonical surjection of $\sim_{\oplus\mu_i}$. The processes $\lcurl \widehat{S^\mu_{i,k}} \rcurl_{k=1}^\infty$ for all $i$ are independent Markov chains.   
\end{proposition}
\begin{proof}
See \autoref{app:decompositionWorks}.
\end{proof}
We can also provide a continuous-time corollary of \autoref{prop:decompositionWorks} that relates the continuous-time controlled compartment CTMDPs to the decomposed uniformised class-aggregated MDPs.  Let $S^\mu(t)$ be the CTMC induced by controlling a stochastically decomposable impulsive CTMDP with a coherent separable policy $\mu$. Let $V^{\mu_i}(t)$ be the stochastic process determined by the $i$th compartment, i.e., the CTMDP $\mathfrak{C}_i$ controlled by $\mu_i$ (note that this is not exactly the same as $[S^\mu(t)]_i$, as updates from $[S^\mu(t)]_i$ to $[S^\mu(t)]_i\oplus\mu_i([S^\mu(t)]_i)$ in the latter can be triggered by events in other compartments, whereas $V^{\mu_i}$ is a wholly independent process not influenced by exogenous information from other compartments). Let $\widehat{V^{\mu_i}}(t)$ be the stochastic process defined by taking the $\sim_{\oplus\mu_i}$ equivalence classes of $V^{\mu_i}(t)$. 
\begin{corollary}
\label{coro:unifToCont}
    The process $ \widehat{V^{\mu_i}}(t)$ is a continuous-time Markov chain that is equivalent in distribution to $\{S_{i,k}^\mu\}_{k=1}^\infty$ under uniformisation.
\end{corollary}
\begin{proof}
    This follows immediately from \autoref{coro:classesAreChainsCont} and inspection of the transition probabilities derived in the proof for \autoref{prop:decompositionWorks}.
\end{proof}
The interpretation of the above corollary is somewhat subtle. First, we discuss the notion of $V^{\mu_i}(t)$ and $[S^\mu(t)]_i$ being distinct. Suppose that a two-compartment CTMDP $\mathfrak{C}$ being controlled by a coherent separable policy $\mu$ is in state $\bm{s}=(s_1,s_2)$, with a non-zero action being taken in both compartments. Upon the next transition, the process moves to a new state of the form $(s_1',s_2\oplus a_2)$ or $(s_1\oplus a_1, s_2')$, so the new state reflects an event in one of the compartments and reflects taking the impulsive action in the other. This is what the whole process $S^\mu(t)$ would tell us, meaning the state of one compartment can be updated to reflect the impulsive action being taken, when the event actually happened in the other compartment. As such, the process $[S^\mu(t)]_1$ is dependent on transition times in $[S^\mu(t)]_2$ and vice versa. However, this is not the case with $V^{\mu_i}(t)$, the compartment CTMDP $\mathfrak{C}_i$ controlled by coherent policy $\mu_i$. Consider $V^{\mu_1}(t)$, the controlled first compartment, and suppose it is in state $s_1$ under action $a_1$. This process cannot be updated to $s_1\oplus a_1$ by some exogenous event, and can only be updated by an actual event in its own compartment. However, when factoring in uniformisation and equivalence relations, $\widehat{V^{\mu_i}}(t)$ is equivalent to $\{S_{i,k}^\mu\}_{k=1}^\infty$ upon uniformisation. This further explains why coherent policies are needed. Without the condition of coherence, when $S^\mu(t)$ transitions from $(s_1,s_2)$ to $(s_1\oplus a_1, s_2')$, one could technically take a new action for $s_1\oplus a_1$, whereas this cannot be done for $V^{\mu_i}(t)$ because no transition is triggered. However, under a coherent policy, this issue never arises.


As these two processes have the same stochastic dynamics, we can without loss of generality consider the random variable $\widehat{S^\mu_i}$, which we define as the unique equilibrium distribution of both $\widehat{V^{\mu_i}}(t)$ and $\{\widehat{S_{i,k}^\mu}\}_{k=1}^\infty$ (note that while this may not exist nor be unique in general, it will exist and be unique given the initial distributions $\widehat{V^{\mu_i}}(0)$ or $\widehat{S_{i,0}^\mu}$, and the guarantee of aperiodicity due to continuous time). These are of course linked to the random variable $\widehat{S^\mu}$, the stationary distribution of the whole-state equivalence class process $\widehat{S^\mu}(t)$. The next corollary relates these two distributions.
\begin{corollary}\label{coro:stationaryIndep}
    
    The random variables $\widehat{S^\mu}$ and $\widehat{S^\mu_i}$ satisfy the following equalities:  
    \begin{align*}
        \PP(\widehat{S^\mu} = \widehat{\bm{s}}) &= \prod_{i=1}^I \PP(\widehat{S^{\mu}_i} = [\widehat{\bm{s}}]_i), \text{ for all }\widehat{\bm{s}}\in\widehat{\mathcal{S}},\\
        \PP(\widehat{S^{\mu}_i} = \widehat{v}) &= \sum_{\widehat{\bm{s}}:[\widehat{\bm{s}}]_i = \widehat{v}}\PP(\widehat{S^\mu} = \widehat{\bm{s}}), \text{ for all }i \text{ and }\widehat{v}\in\widehat{\mathcal{S}_i}.
    \end{align*}
\end{corollary}
\begin{proof}
    By \autoref{coro:unifToCont}, we know the $\widehat{S^{\mu}_i}$ are the stationary distributions of the discrete-time Markov chains $\widehat{S^{\mu}_{i,k}}$, and that the $i$ chains are independent. Firstly, we see:
    \begin{align*}
        \PP(\widehat{S^\mu} = \widehat{\bm{s}}) &= \PP\left(\bigcap_{i=1}^I \left(\left[\widehat{S^\mu}\right]_i = [\widehat{\bm{s}}]_i\right)\right)\\
        &= \prod_{i=1}^I \PP(\widehat{S^\mu_i} = [\widehat{\bm{s}}]_i),\text{ by independence of the }\widehat{S^\mu_i}.
    \end{align*}
    Secondly, we have: 
    \begin{align*}
        \PP(\widehat{S^\mu_i} = \widehat{v}) &= \PP([\widehat{S^\mu}]_i = \widehat{v})\\
        &= \PP(\widehat{S^\mu}\in\{\widehat{\bm{s}}:[\widehat{\bm{s}}]_i = \widehat{v}\})\\
        &= \sum_{\widehat{\bm{s}}:[\widehat{\bm{s}}]_i = \widehat{v}}\PP(\widehat{S^\mu} = \widehat{\bm{s}}).
    \end{align*} 
    \end{proof}
\begin{remark}
In the above corollary we take the ``$i\Th$ compartment of an equivalence class''. Intuitively, this is the same as taking any element of that class, taking its $i\Th$ compartment, then taking the $i\Th$ class of that compartment. A rigorous statement and proof of this can be found in \autoref{app:ithProjectionsAndCanonicalSurjections}. 
\end{remark}
In more familiar notation, if we take $\widehat{\bm{\pi}^\mu}$ to be the $\sim_{\oplus\mu}$ class-aggregated stationary distribution of the controlled whole CTMDP $\mathfrak{C}$ under CS policy $\mu$ , and take $\widehat{\bm{\pi}^{\mu_i}}$ to be the $\sim_{\oplus\mu_i}$ class-aggregated stationary distribution of the controlled compartment CTMDPs $\mathfrak{C}_i$ under CS policies $\mu_i=[\mu]_i$, then we have $\widehat{\pi^\mu}(\widehat{\bm{s}}) = \prod_{i=1}^I \widehat{\pi^{\mu_i}}([\widehat{\bm{s}}]_i)$ and $\widehat{\pi^{\mu_i}}(\widehat{v}) = \sum_{\widehat{\bm{s}}:[\widehat{\bm{s}}]_i = \widehat{v}}\widehat{\pi^\mu}(\widehat{\bm{s}}).$ We can also write this in terms of decision variables for OP. Let $\mu$ be a CS policy, $\pi(\bm{s},\bm{a})$ be the corresponding basic feasible solution to OP for the whole CTMDP $\mathfrak{C}$, and $\pi_i(v,b)$ for $v\in\mathcal{S}_i, b\in\mathcal{A}_i(v)$ be the corresponding basic feasible solution to OP for the $i$th compartment CTMDP $\mathfrak{C}_i$ under policy $[\mu]_i$. Consider the class-aggregated solutions $\widehat{\pi}(\widehat{\bm{s},\bm{a}}) := \sum_{(\bm{s}',\bm{a}')\in\widehat{\bm{s},\bm{a}}}\pi(\bm{s}',\bm{a}')$ under $\sim_\oplus$ and $\widehat{\pi_i}(\widehat{v,b})=\sum_{(v',b')\in\widehat{v,b}}\pi_i(v',b')$ under $\sim_{\oplus_i}$. Immediately from \autoref{lemma:s-aHat} and \autoref{coro:stationaryIndep} we have:
\begin{align}
    \widehat{\pi}(\widehat{\bm{s},\bm{a}}) &= \prod_{i=1}^I\pi_i([\widehat{\bm{s},\bm{a}}]_i), \label{eq-whole-to-compartment-decision-vars}\\
    \widehat{\pi_i}(\widehat{v,b}) &= \sum_{(v',b')\in\widehat{v,b}}\sum_{(\bm{s},\bm{a}):[\bm{s},\bm{a}]_i=(v',b')}\pi(\bm{s},\bm{a}) = \sum_{\widehat{\bm{s},\bm{a}}:[\widehat{\bm{s},\bm{a}}]_i=(v',b')}\widehat{\pi}(\widehat{\bm{s},\bm{a}}). \label{eq-compartment-to-whole-decision-vars}
\end{align}
The above equations will be useful for the analysis in \autoref{paper2:solution}. This concludes the theoretical analysis of decomposable impulsive CTMDPs controlled by coherent separable policies. What remains is for us to develop a scalable solution methodology for the decomposable framework. The state and action spaces of a decomposable impulsive CTMDP grow exponentially in the number of compartments, implying that this framework suffers from the curse of dimensionality. The next section proposes a solution methodology that leverages their decomposable nature.
\section{Approximate Solution Methodology}
\label{paper2:solution}
In this section we theoretically motivate and outline a solution methodology for decomposable impulsive CTMDPs. In \autoref{paper2:method:bd} we show that, with additional constraints that enforce a coherent separable policy, the linear program associated with a decomposable impulsive CTMDP is equivalent to a block-diagonal program with one block per compartment CTMDP. In the single-objective case, this offers an obvious decomposition into separate subproblems that collectively form a solution for the whole problem. In \autoref{paper2:method:dw}, we tackle the multi-objective case and leverage the block-diagonal structure by applying Dantzig-Wolfe decomposition. This yields multi-objective subproblems that are easier to solve, and can be pieced back together to form near-Pareto-optimal solutions to the whole problem using a multi-objective integer program. {\color{black} \autoref{paper2:method:linear} follows by providing linearisations to yield standard integer linear programs, which in turn can be solved by standard MILP solvers.} Recognising the potential computational difficulty of this exact methodology, in \autoref{paper2:method:framework} we propose a relaxation in order to obtain a heuristic solution framework.
\subsection{Block-Diagonalisation}
\label{paper2:method:bd}
The core requirement of our heuristic is that the decomposable impulsive CTMDP will be controlled using a coherent separable policy. In the terminology of \cite{powell2022RLSO_PFA}, we call this a \emph{policy function approximation}, or PFA. In order to find the best such policy, we must add additional constraints to the original problem OP. To achieve this, we must introduce binary variables to avoid non-linear constraints. This yields a problem that we call the Coherent Separable Mixed Integer Linear Program, or CS-MILP. Note that this is not a problem formulation that one would implement using a MILP solver, but a theoretical step towards a decomposition-based method. 
\begin{align}
    &\text{(CS-MILP)} &\min_{\bm{\pi}, \bm{y}} &\,g_1 = \sum_{\bm{s},\bm{a}\in\mathcal{SA}}c_1(\bm{s}\oplus \bm{a})\pi(\bm{s},\bm{a}) \nonumber\\
    &            & \vdots \nonumber\\
    &            &\min_{\bm{\pi},\bm{y}}     &\,g_K = \sum_{\bm{s},\bm{a}\in\mathcal{SA}}c_K(\bm{s}\oplus \bm{a})\pi(\bm{s},\bm{a}) \nonumber\\
    &            &\text{s.t. } & \eqref{op-cons-balance}-\eqref{op-cons-pos} \label{dcimilp-cons-carry}\\
    &            &             & \pi(\bm{s},\bm{a})\leq y(\bm{s},\bm{a}) \text{ for all }(\bm{s},\bm{a})\in\mathcal{SA}, \label{dcimilp-cons-pi-y}\\ 
    &            &             & \sum_{\bm{a}'\in\mathcal{A}(\bm{s}),\bm{a}'\neq \bm{a}}\pi(\bm{s},\bm{a}')\leq 1 - y(\bm{s},\bm{a}), \text{ for all }(\bm{s},\bm{a})\in\mathcal{SA}, \label{dcimilp-cons-det}\\
    &            &             &  \sum_{\bm{a}'\in\mathcal{A}(\bm{s}),\bm{a}'\neq0}\pi(\bm{s}\oplus \bm{a},\bm{a}') \leq 1 - y(\bm{s},\bm{a}), \text{ for all }(\bm{s},\bm{a})\in\mathcal{SA},\\&&&\label{dcimilp-cons-coherent}\\ 
    &            &             &\sum_{\substack{(\bm{s}',\bm{a}')\in\mathcal{SA}:\\ [\bm{s}']_i = [\bm{s}]_i, [\bm{a}']_i\neq[\bm{a}]_i}}\pi(\bm{s}',\bm{a}')\leq 1 - y(\bm{s},\bm{a}),\text{ for all }(\bm{s},\bm{a})\in\mathcal{SA}, \\&&&\label{dcimilp-cons-indep}\\ 
    &            &             & y(\bm{s},\bm{a})\in\{0,1\},\text{ for all }(\bm{s},\bm{a})\in\mathcal{SA}.\label{dcimilp-cons-bin}
\end{align}
We associate with each state-action pair a binary variable $y(\bm{s},\bm{a})$ indicating whether or not our policy uses action $\bm{a}$ in state $\bm{s}$. Constraint \eqref{dcimilp-cons-pi-y} enforces this relation. Constraint \eqref{dcimilp-cons-det} ensures that the policy is deterministic; that is, if $\pi(\bm{s},\bm{a})>0$, then $y(\bm{s},\bm{a})=1$, so this constraint forces $\pi(\bm{s},\bm{a}')=0$ for all $\bm{a}'\neq \bm{a}$. Constraint \eqref{dcimilp-cons-coherent} ensures that the policy is coherent, so if action $\bm{a}$ is used in a recurrent state $\bm{s}$ (i.e., $y(\bm{s},\bm{a})=1$), then we force $\bm{a}'=0$ to be the only allowable action for state $\bm{s}\oplus \bm{a}$. Constraint \eqref{dcimilp-cons-indep} ensures that the policy is separable, so if action $\bm{a}$ is used in state $\bm{s}$, then for all other states $\bm{s}'$ such that $[\bm{s}']_i = [\bm{s}]_i$ we ensure that the allowable actions $\bm{a}'$ satisfy $[\bm{a}']_i=[\bm{a}]_i$. 

We wish to ``block-diagonalise'' CS-MILP by splitting it into $I$ independent CTMDPs. To do this, we define a new problem that we call the Decomposed Coherent Mixed Integer Program, or DC-MIP. Again, this is not a formulation to be solved directly, but a theoretical tool. 
\begin{align}
    \text{(DC-MIP)} \quad&\min_{\bm{\pi}_i,\bm{y}_i} \,g_1=f_1\left(\left(\sum_{(s_i,a_i)\in\mathcal{S}\mathcal{A}_i}c_{i,1}(s_i\oplus a_i)\pi_i(s_i,a_i)\right)_{i=1}^I\right)\\
        & \quad\vdots \nonumber\\
        &\min_{\bm{\pi}_i,\bm{y}_i} \,g_K=f_K\left(\left(\sum_{(s_i,a_i)\in\mathcal{S}\mathcal{A}_i}c_{i,K}(s_i\oplus a_i)\pi_i(s_i,a_i)\right)_{i=1}^I\right)\\
        &\text{s.t. for each $i\in\{1,\dots,I\}$} \nonumber\\
        &\quad\sum_{(s_i,a_i)\in\mathcal{SA}_i}q_i(s_i,a_i,s'_i)\pi_i(s_i,a_i) = 0,\text{ for all }s'_i\in\mathcal{S}_i, \label{ddcmilp-cons-balance}\\             &\quad\sum_{(s_i,a_i)\in\mathcal{SA}_i}\pi_i(s_i,a_i) = 1, \label{ddcmilp-cons-norm}\\
        &\quad\pi_i(s_i,a_i)\geq 0,\text{ for all }(s_i,a_i)\in\mathcal{SA}_i. \label{ddcmilp-cons-pos}\\
        &\quad \pi_i(s_i,a_i)\leq y_i(s_i,a_i), \text{ for all }(s_i,a_i)\in\mathcal{SA}_i, \label{ddcmilp-cons-y-pi}\\
        &\quad \sum_{a'_i\in\mathcal{A}_i(s_i),a'_i\neq a_i}\pi_i(s_i,a'_i)\leq 1 - y_i(s_i,a_i),\text{ for all }(s_i,a_i)\in\mathcal{SA}_i, \label{ddcmilp-cons-det}\\
        &\quad \sum_{a'_i\in\mathcal{A}_i(s),a'_i\neq0}\pi_i(s_i\oplus a_i,a'_i) \leq 1 - y_i(s_i,a_i),\text{ for all }(s_i,a_i)\in\mathcal{SA}_i,\\&\label{ddcmilp-cons-coherent}\\
        &\quad y_i(s_i,a_i)\in\{0,1\},\text{ for all }(s_i,a_i)\in\mathcal{SA}_i.\label{ddcmilp-cons-bin}
\end{align}
Note that in the objective functions of the DC-MIP, the $\pi_i$ refer to the state-action frequencies of the compartment CTMDPs, and are not defined in terms of the global state-action frequencies as in the definition of the decomposable cost function. For each of the $i$ compartment CTMDPs, constraints \eqref{ddcmilp-cons-balance} give the balance equations, \eqref{ddcmilp-cons-norm} and \eqref{ddcmilp-cons-pos} ensure $\pi_i$ is a valid probability distribution, \eqref{ddcmilp-cons-y-pi} ensures $y_i(s_i,a_i)$ is 1 if action $a_i$ is used in state $s_i$ if $s_i$ is recurrent, \eqref{ddcmilp-cons-det} ensures that policies are deterministic, \eqref{ddcmilp-cons-coherent} ensures that policies are coherent, and \eqref{ddcmilp-cons-bin} ensures that the $y_i$ variables are binary. 

We observe that there are no linking constraints between any of the $i$ sub-CTMDPs, so this problem has a strict block-diagonal structure, allowing for an easy Dantzig-Wolfe decomposition. However, before this, it must be shown that the two problems CS-MILP and DC-MIP are in some way equivalent. The sense of equivalence we require is that for any solution in one programme, there exists a solution to the other programme with the same objective values. In order to prove this, we will show that it is possible to move freely between the space of solutions to either programme (i.e., the mixed integer programming perspective) in such a way that objective values are preserved. As this result is somewhat expected, we simply state the theorem here and defer the proof to the appendix.

\begin{theorem}
\label{thm:CS-MILP_and_DC-MIP_equiv}
    There exists a solution to CS-MILP with objective value $\bm{g}=\bm{\Bar{g}}$ if and only if there exists a solution to DC-MIP with objective value $\bm{g}=\bm{\Bar{g}}$.
\end{theorem}
\begin{proof}
See \autoref{app:CS-MILP_and_DC-MIP}.
\end{proof}
As such, further analyses need only consider the decomposed problem DC-MIP, as every solution for CS-MILP has an equivalent solution in DC-MIP and vice versa. In the single-objective case, this completes our analysis, as each block of DC-MIP can be optimised separately alongside its decomposed cost function, the integer variables and additional constraints can be omitted and the block reduces to an instance of OP, which can be solved via LP, value iteration or other methods. This will always result in an optimal DS policy which, if needed, can be sequenced to obtain a coherent policy. The coherent policies together then give the coherent policy for the whole system. However, in the multi-objective case, the task of putting the policies together is non-trivial, as we have a set of Pareto-optimal solutions per compartment, and these must be pieced together in order to obtain non-dominated solutions for the whole problem. The next subsection explores how to do this.
\subsection{Dantzig-Wolfe Decomposition}
\label{paper2:method:dw}
To aid in the solution of DC-MIP, we can leverage the block-diagonal structure of the problem and apply \emph{Dantzig-Wolfe decomposition}. Under such a decomposition for a continuous linear problem of the form $\min_{\bm{x}}\{f(\bm{x}):A\bm{x}=\bm{b},x\geq 0\}$ with a block-diagonal component, we recognise that each block of the block-diagonal matrix defines a polytope in the decision space. For each block, we can replace the decision variables with \emph{disaggregated} variables $\lambda_s^i\in[0,1]$, each of which corresponds to a corner of the polytope, or equivalently a basic feasible solution of the subproblem corresponding to that block. By convexity of the feasible region, any solution can be recovered by a weighted sum of the polytope corners, with weights summing to one. When considering DC-MIP, due to the binary variables, the polytope is the convex hull of the integral solutions to each block. Each basic feasible solution of that polytope will correspond to a coherent policy $\mu_i$ that induces a unichain CTMC. Some of these solutions may even correspond to the same policy, due to symmetries on the $y$ variables for transient states. We also recognise that for each coherent policy $\mu_i$, we have corresponding long-run average costs $g_{i,k}(\mu_i) = \sum_{s_i\in\mathcal{S}_i}c_{i,k}(s_i\oplus\mu_i(s_i))\pi_i(s_i,\mu_i(s_i))$. We therefore construct the following decomposed coherent master problem, or DC-MP:
\begin{align}
    \text{(DC-MP)}\quad\min_{\bm{\lambda}}&\,g_1 = f_1\left(\left(\sum_{\mu_i\in\mathcal{M}_i^{DC}}g_{i,1}(\mu_i)\lambda_{\mu_i}^i\right)_{i=1}^I\right) \nonumber \\
                   &\vdots \nonumber \\
                   \min_{\bm{\lambda}}&\,g_K = f_K\left(\left(\sum_{\mu_i\in\mathcal{M}_i^{DC}}g_{i,K}(\mu_i)\lambda_{\mu_i}^i\right)_{i=1}^I\right)\\
                   \text{s.t.}& \text{ for every }i\in\{1,\dots,I\}: \nonumber \\
                   & \sum_{\mu_i\in\mathcal{M}_i^{C}}\lambda_{\mu_i}^i = 1,\\
                   & \lambda_{\mu_i}^i\in\{0,1\},\text{ for all }i\in[I],\mu_i\in\mathcal{M}_i^{C},
\end{align}
where $\mathcal{M}_i^{C}$ is the set of all unichain coherent policies for compartment $i$. We constrain the disaggregated variables $\bm{\lambda}$ to be binary instead of $[0,1]$ as we do not have linking constraints to be concerned about, and we want to keep our solutions deterministic. In this initial formulation, the number of binary variables is very large; however, this can be alleviated by identifying a much smaller subset of variables to be included, if it can be guaranteed that the resulting formulation still finds all Pareto-optimal solutions.
\begin{proposition}\label{prop:notPFimpliesNotPF}
    Let $\mathfrak{C}$ be a decomposable impulsive CTMDP with compartments $\mathfrak{C}_i$. A solution to DC-MP is Pareto-optimal only if the compartment policies $\mu_i$ are Pareto-optimal with respect to the vector objective functions $\bm{g}_i$.
    If a policy $\mu = (\mu_1,\dots,\mu_I)$ has sub-policies $\mu_{i}$ that are not Pareto-optimal, then there exists a solution $\mu'$ with Pareto-optimal sub-policies such that the objective values of $\mu$ are bounded below by $\mu'$.
\end{proposition}
\begin{proof}
    Construct $\mu'$ as follows: let $\mu'_i = \mu_i$ if $\mu_i$ is Pareto-optimal; otherwise, choose a Pareto-optimal $\mu_i'$ that dominates $\mu_i$. Now $\bm{g}_i(\mu_i')\leq\bm{g}_i(\mu_i)$ component-wise, and as the $f_k$ are increasing functions, $f_k(g_{1,k}(\mu'_1),\dots,g_{I,k}(\mu'_I)) \leq f_k(g_{1,k}(\mu_1),\dots,g_{I,k}(\mu_I))$ for all $k$.
\end{proof}
By \autoref{prop:notPFimpliesNotPF}, we can reduce the number of decision variables by including only the ones that correspond to a Pareto-optimal compartment policy. We write $\PF(\mathcal{M}_i^{C})$ to be the set of Pareto-optimal polices in $\mathcal{M}_i^{C}$ with respect to $\bm{g}_i$, i.e., the Pareto front. As such we can write the reduced master problem, which we call DC-RMP:
\begin{align}
    \text{(DC-RMP)}\quad\min_{\bm{\lambda}}&\,g_1 = f_1\left(\left(\sum_{\mu_i\in\PF(\mathcal{M}_i^{DC})}g_{i,1}(\mu_i)\lambda_{\mu_i}^i\right)_{i=1}^I\right) \nonumber \\
                   &\vdots \nonumber \\
                   \min_{\bm{\lambda}}&\,g_K = f_K\left(\left(\sum_{\mu_i\in\PF(\mathcal{M}_i^{DC})}g_{i,K}(\mu_i)\lambda_{\mu_i}^i\right)_{i=1}^I\right)\\
                   \text{s.t.}& \text{ for every }1\leq i\leq I: \nonumber \\
                   & \sum_{\mu_i\in\PF(\mathcal{M}_i^{C})}\lambda_{\mu_i}^i = 1,\\
                   & \lambda_{\mu_i}^i\in\{0,1\},\text{ for all }\mu_i\in\PF(\mathcal{M}_i^{C}).
\end{align}
By \autoref{prop:notPFimpliesNotPF}, any Pareto-optimal solution to DC-MP has an equivalent in DC-RMP, and vice versa. If $\PF(\mathcal{M}_i^{DC})$ can be obtained exactly, then we can obtain Pareto-optimal solutions for DC-RMP and in turn for CS-MILP using this decomposition. However, in practice, finding $\PF(\mathcal{M}_i^{DC})$ exactly is difficult, as it refers to the set of DS policies that are Pareto-optimal only with respect to the other DS policies. This means the set contains policies that are not optimal for any mixed-objective scalarisation of the objective function, and are therefore hard to find. As such, we instead consider subsets $\overline{\PF}(\mathcal{M}_i^{DC})$ of the Pareto front which are representative of the front, but which are easier to find.

As currently presented, DC-RMP may be a non-linear MIP due to the $f_k$ function, making it hard to solve. Ideally, we would want an MILP formulation to be solved by an off-the-shelf MILP solver. For the decomposable cost functions discussed in this paper, the following subsection discusses some linearisations.

\subsection{Linearisations}
\label{paper2:method:linear}
We can linearise DC-RMP if we can apply a strictly increasing transformation $T_k$ to each $f_k$ such that $(T_k\circ f_k)(\dots,\sum_{\mu_i\in\PF(\mathcal{M}_i^{DC})}g_{i,k}\lambda_{\mu_i}^i,\dots)$ is a linear combination of the $\lambda_{\mu_i}^i$ whenever the $\lambda_{\mu_i}^i$ satisfy the constraints, as we can then use any MILP solver to solve DC-RMP. This is because, in general for a multi-objective optimisation problem with decision variables $\bm{x}\in \mathcal{X}$, a solution $\bm{x}$ is Pareto-optimal for objectives $(f_1,\dots,f_K)$ if and only if it is also Pareto-optimal for $(T_1\circ f_1,\dots,T_K\circ f_k)$ where the $T_k$ are strictly increasing functions. Such transformations exist for all additive, multiplicative, and logical-OR cost functions. Of course, this is trivially evident for the additive costs case, with $T_k(x) = x$. In the case of multiplicative costs, we choose $T_k(x) = \ln(x)$ and recall $f_k(\bm{x})=\prod_{i=1}^Ix_i$, and find:
\begin{align*}
    (T_k \circ f_k)\left(\left(\sum_{\mu_i\in\PF(\mathcal{M}_i^{C})}g_{i,k}(\mu_i)\lambda_{\mu_i}^i\right)_{i=1}^I\right) &= \ln\lcurl \prod_{i=1}^I \sum_{\mu_i\in\PF(\mathcal{M}_i^{DC})}g_{i,k}(\mu_i)\lambda_{\mu_i}^i\rcurl\\
    &=\sum_{i=1}^I \ln\lcurl \sum_{\mu_i\in\PF(\mathcal{M}_i^{C})}g_{i,k}(\mu_i)\lambda_{\mu_i}^i\rcurl\\
    &= \sum_{i=1}^I \sum_{\mu_i\in\PF(\mathcal{M}_i^{C})}\ln\lcurl g_{i,k}(\mu_i)\rcurl\lambda_{\mu_i}^i,
\end{align*}
where the last equality holds only for binary $\lambda_{\mu_i}^i$ such that $\sum_{\mu_i\in\PF(\mathcal{M}_i^{C})}\lambda_{\mu_i}^i = 1$. Similarly, for the logical-OR case, we choose $T_k(x) = -\ln(1-x)$ and recall $f(\bm{x})=1 - \prod_{i=1}^I(1-x_i)$ and find:
\begin{align*}
    &(T_k \circ f_k)\left(\left(\sum_{\mu_i\in\PF(\mathcal{M}_i^{C})}g_{i,k}(\mu_i)\lambda_{\mu_i}^i\right)_{i=1}^I\right)\\ &= -\ln\lcurl 1 - \left( 1 - \prod_{i=1}^I \left( 1 - \sum_{\mu_i\in\PF(\mathcal{M}_i^{C})} g_{i,k}(\mu_i)\lambda_{\mu_i}^i\right)\right)\rcurl\\
    &= - \sum_{i=1}^I \ln \lcurl 1 - \sum_{\mu_i\in\PF(\mathcal{M}_i^{C})} g_{i,k}(\mu_i)\lambda_{\mu_i}^i\rcurl\\
    &= -\sum_{i=1}^I \sum_{\mu_i\in\PF(\mathcal{M}_i^{C})} \ln\lcurl1-g_{i,k}(\mu_i)\rcurl\lambda_{\mu_i}^i,
\end{align*}
where again the equality holds only for suitable $\lambda_{\mu_i}^i$. Therefore, if all objectives are additive, multiplicative, or logical-OR cost functions, DC-RMP can be reformulated as having linear objective functions, and is therefore a mixed integer linear program that can be solved by any MILP solver. 

\subsection{Solution Framework}
\label{paper2:method:framework}
As it stands, obtaining the full Pareto fronts $\text{PF}(\mathcal{M}_i^C)$ for the subproblems and plugging them into DC-RMP provides an exact solution method for CS-MILP and DC-MIP. However, it does not necessarily follow that such solutions are Pareto-optimal for the original problem OP, as we have made the simplifying assumption throughout that we limit ourselves to separable policies. There are also some remaining problems with this approach. Obtaining $\PF(\mathcal{M}_i^{C})$ can be challenging in its own right even for small CTMDPs, and the set itself can still have a high cardinality and result in an instance of DC-RMP with a large number of binary variables. This is especially true if there's a large number of objectives, leading to more policies being Pareto-optimal. Also, in general, even single-objective MDPs can be difficult or impossible to solve exactly due to the curse of dimensionality. 

To convert our exact methodology into a flexible, scalable approximate methodology, we may simply replace the exact Pareto fronts $\text{PF}(\mathcal{M}_i^C)$ in DC-RMP with \emph{approximate} Pareto fronts $\overline{\text{PF}}(\mathcal{M}_i^C)$, obtaining DC-RMP-Approx. There are two main ways that the solution sets $\overline{\PF}$ are approximate: they may be incomplete (i.e., missing Pareto-optimal solutions), or the solutions present may not be Pareto-optimal. Nevertheless, if the solution set is reasonably representative of the objective space (i.e., captures the different balances between the different objectives), and the extent of suboptimality is acceptable, this method may still be worthwhile in practice. We now present our solution methodology as a meta-algorithm:
\begin{enumerate}
    \item Choose a decomposition $\{\mathfrak{C}_i\}_{i=1}^I$ of the decomposable impulsive CTMDP $\mathfrak{C}$, and identify the compartment cost functions $c_{i,k}$ for each compartment and objective.
    \item Obtain an exact or approximate Pareto front of solutions to each compartment subproblem, taking care to ensure all policies are coherent. This may be done with exact methods by scalarising the problem (e.g. using the weighted-sum) and then using methods such as value iteration or linear programming, or by using approximate methods such as multi-objective reinforcement learning (MORL) \citep{van2014morl} or bespoke heuristics. Where non-coherent policies are found, apply sequencing to obtain coherent policies (but recall that this may result in further suboptimality when the non-coherent policies are suboptimal).
    \item Use the exact fronts or approximate fronts as the subproblem solution sets in DC-RMP or DC-RMP-Approx (respectively), and use a scalarising method together with a MILP solver (or some heuristic solution method) to obtain the Pareto fronts for these problems.
\end{enumerate}
This meta-algorithm allows for a lot of flexibility, including the choice of decomposition, the solution methods for the subproblems, and the handling of the multiple objectives in the reduced master problem.

It remains for us to explain why one might use this solution methodology as opposed to other popular approximate methodologies that could be applied directly to decomposable impulsive CTMDPs, such as MORL. In the case where compartment MDPs are small enough to be solved exactly, this method benefits from the true optimality of the subproblems, and the lack of complications that arise from reinforcement learning methods (such as the choice of value function approximation or neural network architecture, parameter tuning, etc). On the other hand, in the case where compartment MDPs cannot be solved exactly, we argue that it may be easier to apply MORL methods or design bespoke heuristics for the simpler compartment CTMDPs than for the whole problem. In either case, our method can also serve as a benchmark or as a warm start for reinforcement learning methods. For example, we could provide policies from our approximate Pareto front as the baseline policies for methods such as approximate policy iteration or rollout \citep{bertsekas2021rollout}, or as the initial ``actor'' in an actor-critic method \citep{konda1999actor}.

In this section, we have developed a flexible decomposition-based heuristic solution methodology that can be used to approximately solve decomposable impulsive CTMDPs with multiple objectives, and achieved this using Dantzig-Wolfe decomposition. In the next section, we consider an application of the decomposable impulsive CTMDP framework, and demonstrate the proposed solution methodology.
\section{Application}
\label{paper2:app}
In this section we demonstrate an application of our methodology to a bi-objective dynamic maintenance problem, centred on a series-parallel system. Such a system is comprised of multiple subsystems connected in series, with each subsystem made up of redundant components connected in parallel. A subsystem is considered functional if at least one of its components is functional, but the whole system is considered functional only if all of the subsystems are functional. We wish to find Pareto optimal solutions with respect to maintenance costs and reliability. This problem can be thought of as a natural extension of the redundancy allocation problem (RAP), which is an NP-hard non-linear combinatorial optimisation problem arising from reliability engineering, where redundant copies of components can be added to a system and the system only fails when all copies of a critical component are damaged \citep{CHERN1992RAP_NP_HARD}. Of particular relevance to this study are repairable series-parallel systems that make use of Markovian modelling to model system degradation and repairs. \cite{KAYEDPOUR2017repairableMarkovFiniteHorizon} and \cite{tavana2018multistate} both make use of CTMCs in modelling the degradation and repair processes of series-parallel systems, where the former uses binary state modelling (healthy or damaged, with nothing in between), and the latter uses multi-state modelling, allowing for a more granular scale of degradation states. Both models work under the assumption that components are always repaired immediately. We extend the binary state model by relaxing the assumption that components are always repaired and instead suppose that decisions about repairs are made as part of a bi-objective sequential decision making problem under uncertainty. This yields a bi-objective CTMDP, in which it is natural to think of the subsystems with parallel redundancy as decomposable compartments.

In \autoref{paper2:app:probDef} we formulate our problem by first defining the compartment CTMDPs, then taking the direct product of these and choosing the cost functions. In \autoref{paper2:app:method} we discuss the specific implementation of our methodology to this problem. \autoref{paper2:app:results} then provides a comparison between an exact method and our heuristic decomposition-based method in terms of runtime, optimality, and number of solutions in the Pareto fronts. We also apply our methodology to larger instances that cannot be solved exactly in order to demonstrate the scalability of our solution methodology.
\subsection{Problem Definition}
\label{paper2:app:probDef}
We construct our CTMDP using the direct product method, which begins by defining the structures of the compartment CTMDPs with components connected in parallel and then uses these to construct the whole series-parallel CTMDP. Recall that a cost-free impulsive CTMDP is defined as a tuple $\mathfrak{C}_i = \langle\mathcal{S}_i, \left(\mathcal{A}_i(s_i)\right)_{s_i\in\mathcal{S}_i}, 
q^+:\mathcal{S}_i\times\mathcal{S}_i\to\RR_{\geq0}, \oplus_i:\mathcal{SA}_i\to\mathcal{S}_i\rangle.$ We assume that the components in a subsystem are homogeneous, so there is no need to differentiate between specific components. We assume that each component can be in one of three conditions: healthy, damaged, or repairing. A healthy component is one that can be used, a damaged component cannot be used and is not undergoing maintenance, and a repairing component cannot be used but is undergoing maintenance. Including a repairing condition in the state is based on the assumption that after a repair has been started, it must be seen through to completion. Without this inclusion, a repair that is started in one state could be discontinued after a transition. Due to the homogeneity of components within a subsystem, it is sufficient for the state to include the numbers of components that are damaged and repairing, rather than detailing specifically which copies are in which condition.

Suppose that a subsystem $i$ has $M_{i}$ components. The sub-state must give the number of copies of the component that are currently repairing and the number damaged. The number that are healthy can then be inferred from the total number of components $M_i$. We define $\mathcal{S}_{i} := \{(s_{i,1},s_{i,2})\in\ZZ_{\geq0}^2: s_{i,1}+s_{i,2}\leq M_{i}\},$ which gives the set of all possible combinations of the number of components in subsystem $i$ that are repairing ($s_{i,1}$) or damaged ($s_{i,2}$). For notational convenience, we write $s_{i,0}:=M_{i}-s_{i,1}-s_{i,2}$ as the number of healthy components in subsystem $i$. For the action spaces, we define: $\mathcal{A}_{i}(s_{i}):= \{a\in\ZZ:0\leq a\leq s_{i,2}\},$ so the number of components in subsystem $i$ that can be put into repair is bounded above by the number of damaged components in that subsystem. For the impulsive operator, we define $\oplus_{i}:\mathcal{SA}_{i}\to\mathcal{S}_{i}$ as $\bm{s}_{i}\oplus_{i}\bm{a}_{i} = (s_{i,1}+a_{i},s_{i,2}-a_{i})$. Finally, we must define the positive transition rates $q_i^+$. There are two types of event which cause a state transition:
\begin{itemize}
    \item \textbf{Degradation} - The degradation of a component due to an random event. The interarrival times of such events are assumed to be independent and exponentially distributed with rate $\alpha_{i} > 0$ for subsystem $i$. 
    \item \textbf{Successful Repair} - The event that changes the condition of a repairing component to `healthy'. We assume that repair completion times are independent and exponentially distributed with rate $\tau_{i} > 0$ for subsystem $i$.
\end{itemize}
We can then construct the positive transition rates $q^+$: 
$$q^+_i(\bm{s}_i, \bm{s}'_i) = 
\begin{cases}
     s_{i,0}\alpha_{i}, &\bm{s}'_i = \bm{s}_i + \bm{e}_{i,2},\\
     s_{i,1}\tau_{i} , &\bm{s}'_i = \bm{s}_i - \bm{e}_{i,1},\\
     0, & \text{otherwise,}
\end{cases}$$
where $\bm{e}_{i,k}$ is a 2-tuple with $\bm{e}_{i,1}=(1,0)$ and $\bm{e}_{i,2}=(0,1)$. The first line gives the rate of a healthy component in subsystem $i$ failing and becoming damaged and the second line gives the rate of a repairing component in subsystem $i$ completing its repair and becoming healthy. 

We can also construct the vector costs of the compartment CTMDPs. The first cost reflects the resources required by ongoing repairs. If a component in subsystem $i$ is currently under repair then we incur a cost at rate $r_{i}$ per unit time. Our second cost is an indicator function for system failure, where no healthy components are available. Symbolically, we have an operational cost rate
$c^o_i(\bm{s}_i) = s_{i,1}r_i$ and a failure indicator cost rate
$c^f_i(\bm{s}_i) = \II\{s_{i,0} = 0\}$.

To create the cost-free version of the series-parallel model, under the assumption that no event in one subsystem would cause an event or change in another, we take the direct product of the subsystems $\mathfrak{C}_i$ and let $\mathfrak{C}=\bigotimes_{i=1}^I\mathfrak{C}_i$ denote the cost-free decomposable impulsive CTMDP for a series-parallel system. The operational cost of the system is defined as the sum of the operational costs of the subsystems, and hence $c^o(\bm{s}) = \sum_{i=1}^I c_o^i([\bm{s}]_i)$. Clearly, this cost is decomposable, as per \autoref{prop:addMultOR_are_decompable}. The cost of system failure is slightly different, because a series-parallel system fails whenever at least one subsystem fails. As such, we have $c^f(\bm{s}) = \II\{\bigvee_{i=1}^I c_i^f([\bm{s}]_i) = 1\}$, which is decomposable by \autoref{prop:addMultOR_are_decompable} because it is a logical-OR cost function.
\subsection{Methodology}
\label{paper2:app:method}
The aim of our approach is to find a well-populated set of Pareto-optimal, or approximately Pareto-optimal, solutions for our dynamic maintenance problem. This does not necessarily mean finding the entire Pareto front, as this can be computationally very expensive, hard to verify, and can result in a large number of very similar solutions.
\subsubsection{Exact Method}
The first approach, and our point of comparison for the decomposition-based heuristic, is an exact method based on the original problem OP. We scalarise the problem by using the mixed-objectives method, and consider a parameterised single objective of the form 
$$
  g = \sum_{(\bm{s},\bm{a})\in\mathcal{SA}}\left(c^o(\bm{s},\bm{a}) + Pc^f(\bm{s},\bm{a})\right)\pi(\bm{s},\bm{a}).
$$ 
The parameter $P$ acts as a penalty for system failure. We expect low values of $P$ to result in more passive maintenance policies, and high values to result in more active policies. A sufficiently high value of $P$ should yield the policy that always repairs all components as soon as they become damaged. As this is a mixed-objective problem with positive objective weights, solutions to this for any value of $P$ will be Pareto-optimal. We call the resulting problem the scalarised original problem, or S-OP.
\begin{align*}
    &\text{(S-OP)} &\min_{\bm{\pi}}   &\,g= \sum_{(\bm{s},\bm{a})\in\mathcal{SA}}\left(c^o(\bm{s},\bm{a}) + Pc^f(\bm{s},\bm{a})\right)\pi(\bm{s},\bm{a})\\
    &            & \text{s.t.} &\sum_{(\bm{s},\bm{a})\in\mathcal{SA}}q(\bm{s},\bm{a},\bm{s}')\pi(\bm{s},\bm{a}) = 0,\text{ for all }\bm{s}'\in\mathcal{S}, \\
    &            &             &\sum_{(\bm{s},\bm{a})\in\mathcal{SA}}\pi(\bm{s},\bm{a}) = 1,\\
    &            &             &\pi(\bm{s},\bm{a})\geq 0,\text{ for all }(\bm{s},\bm{a})\in\mathcal{SA}.
\end{align*}

The above problem can be solved for any value of $P$ using an LP solver. For implementation details on using the LP solver and selecting values of $P$, see Supplementary Material 2.1.  
\subsubsection{Heuristic Method}
The second approach is to use our decomposition-based heuristic. To do this, we must obtain an approximate Pareto front $\overline{\PF}(\mathcal{M}_i^{DC})$ for each of the sub-CTMDPs $\mathfrak{C}_i$. We can do this by following the exact method described above, scalarising the bi-objective problem using a penalty parameter $P$ and re-solving the problem for a range of values. For implementation details, see Supplementary Material 2.2. The solutions obtained from this method can be recombined via the reduced master problem DC-RMP. After applying the suitable linearising transformation to the second objective function, we obtain the following bi-objective approximate reduced master problem, which we call BO-ARMP:
\begin{align*}
    \text{(BO-ARMP)}\quad\min_{\bm{\lambda}} &\, g^o = \sum_{i=1}^I \sum_{\mu_i\in\overline{\PF}(\mathcal{M}_i^{DC})}g_{i}^o(\mu_i)\lambda_{\mu_i}^i\\
    \min_{\bm{\lambda}} &\, T(g^f) = \sum_{i=1}^I \sum_{\mu_i\in\overline{\PF}(\mathcal{M}_i^{DC})}-\ln\lcurl1-g_i^f(\mu_i)\rcurl\lambda_{\mu_i}^i\\
    \text{s.t.}  &\text{ for every }1\leq i \leq I:\\
                 &\sum_{\mu_i\in\overline{\PF}(\mathcal{M}_i^{DC})}\lambda_{\mu_i}^i=1,\\
                 &\lambda_{\mu_i}^i\in\{0,1\}, \text{ for all }\mu_i\in\overline{\PF}(\mathcal{M}_i^{DC}).
\end{align*}

To handle the bi-objective optimisation, we will use the $\varepsilon-$constraint method, and turn the second objective into a constraint. This yields the $\varepsilon-$constrained approximate reduced master problem, or $\varepsilon-$ARMP:
\begin{align*}
    \text{($\varepsilon$-ARMP)}\quad\min_{\bm{\lambda}} &\, g^o = \sum_{i=1}^I \sum_{\mu_i\in\overline{\PF}(\mathcal{M}_i^{DC})}g_{i}^o(\mu_i)\lambda_{\mu_i}^i\\
    \text{s.t.}  &\text{ for every }1\leq i \leq I:\\
                 &\sum_{\mu_i\in\overline{\PF}(\mathcal{M}_i^{DC})}\lambda_{\mu_i}^i=1,\\
                 &\lambda_{\mu_i}^i\in\{0,1\} \text{ for all }\mu_i\in\overline{\PF}(\mathcal{M}_i^{DC}),\\
                 &\sum_{i=1}^I \sum_{\mu_i\in\overline{\PF}(\mathcal{M}_i^{DC})}-\ln\lcurl1-g_i^f(\mu_i)\rcurl\lambda_{\mu_i}^i \leq \varepsilon.
\end{align*}
This yields a scalarised problem which can be solved via any MILP solver. As in the case of the penalty values $P$, we now need to choose values of $\varepsilon$ to be optimised against. We first choose a starting value for the target probability of failure $g^f_{start}$, then convert this via the linearising transformation, so $\varepsilon_{start} = -\ln(1 - g^f_{start})$. Solving the problem with $\varepsilon = \varepsilon_{start}$ gives a solution that we call $\lambda^{(1)}$, which has second objective value $T(g^{f(1)}) \leq \varepsilon_{start}$. Using a parameter $0 < \gamma < 1$, we can set $\varepsilon = \gamma T(g^{f(1)})$ and use this value of $\varepsilon$ to ``cut off'' the previous solution, forcing the programme to find a solution with a better second objective value. Solving this programme gives a solution $\lambda^{(2)}$ with second objective value $T(g^{f(2)})$, and then the process repeats. This continues until the algorithm either finds the ``always repair'' solution or the programme becomes infeasible.
\subsection{Computational Results}
\label{paper2:app:results}
We consider problem instances derived from subsets of the solution to the 14-subsystem redundancy allocation problem considered by  \cite{fyffe1968problemset}. Each component has a long-run reliability value $p_i$, interpreted as the proportion of time that the component is healthy, assuming it is always repaired immediately. However, repair rates, degradation rates and repair costs are not included in the problem considered by the authors. For the sake of simplicity, we assume that all components have the same repair rate $\tau = 1$ and the same repair cost rate $r = 100$. Given the reliability values $p_i$ for each component, we then choose the degradation rate $\alpha_i$ such that $\tau_i/(\tau_i + \alpha_i) = p_i$, where $\tau_i/(\tau_i + \alpha_i)$ is the long-run proportion of up-time for a component that is always repaired as soon as it becomes damaged.  All LPs are solved using Gurobi 10.0.2 \citep{gurobi}, and all code is written in Julia 1.9.3. The code is run on a computer with an Intel Xeon Gold 6348 with four cores clocked at 2.6 GHz and 16 GB of RAM.
\begin{table}[tbp]
    \centering
    \begin{tabular}{ccccc}\toprule
    Subsystem & $p$ & $\tau$ & r & Copies\\
    \midrule
    1         & 0.91 & 1.0 & 100.0 & 3\\
    2         & 0.95 & 1.0 & 100.0 & 2\\
    3         & 0.92 & 1.0 & 100.0 & 3\\
    4         & 0.85 & 1.0 & 100.0 & 3\\
    5         & 0.93 & 1.0 & 100.0 & 3\\
    6         & 0.98 & 1.0 & 100.0 & 2\\
    7         & 0.91 & 1.0 & 100.0 & 2\\
    8         & 0.81 & 1.0 & 100.0 & 4\\
    9         & 0.96 & 1.0 & 100.0 & 2\\
    10        & 0.85 & 1.0 & 100.0 & 3\\
    11        & 0.94 & 1.0 & 100.0 & 2\\
    12        & 0.79 & 1.0 & 100.0 & 4\\
    13        & 0.99 & 1.0 & 100.0 & 2\\
    14        & 0.95 & 1.0 & 100.0 & 2\\
    \bottomrule
    \end{tabular}
    \caption{Subsystem designs and parameters}
    \label{tab:14subsystems}
\end{table}
We obtain subproblems by taking contiguous ranges of subsystems from these 14, and treating them as one system. For example, the problem we call 1:4 is a system comprised of subsystems 1, 2, 3 and 4. For the exact method, we use penalty values $P_n=n$ for $n = 1,2,\dots,10000$. For the heuristic method, we use penalty values $P_n = 10^{n/10}$ for $n = 1,2,\dots,n_{\max}=50$. We have found that it is only possible to solve instances exactly if there are no more than four subsystems; all 5-subsystem instances either ran into memory issues whilst building the model and were unable to begin solving, or did not find any solutions except for the ``never repair'' solution in reasonable time. Specifically, for those 5-subsystem instances that could fit in memory, no other solutions were found within 24 hours. 

In order to compare the exact (albeit incomplete) Pareto front with the approximate Pareto front found by our method, in \autoref{tab:exactVsApprox} we report summary statistics of an error measurement that we call the \emph{Relative Euclidean Pareto Optimality Gap}, or REPOG. If a vector $\bm{z^*}$ is a Pareto-optimal point in $K$-dimensional objective space and $\bm{z}$ is a point dominated by $\bm{z^*}$, then the objective-wise relative optimality gaps (ROGs) are given by the vector: $$\text{ROG}(\bm{z}|\bm{z^*})= \begin{cases}
    \left(\left|\frac{z_1-z^*_1}{z^*_1}\right|,\dots,\left|\frac{z_K-z^*_K}{z^*_K}\right|\right), & \bm{z^*} \text{ dominates }\bm{z}, \\
     \text{undefined,}& \text{otherwise.} 
\end{cases}$$
REPOG is then simply given by the Euclidean norm of ROG, so:
$$\text{REPOG}(\bm{z}|\bm{z^*}) = ||\text{ROG}((\bm{z}|\bm{z^*})||.$$ 
Note that in cases where $\bm{z}$ is equal to $\bm{z^*}$ in all but one dimension, or where the dimension of the objective space is 1, then REPOG reduces to the standard relative optimality gap. More generally, given a set of points $\mathcal{Z}^*$ all of which dominate $\bm{z}$, we can define:
$$\text{REPOG}(\bm{z}|\mathcal{Z}^*) = \min_{\bm{z^*}\in \mathcal{Z}^*}\text{REPOG}(\bm{z}|\bm{z^*}).$$
If the set $\mathcal{Z}^*$ gives the set of all points in the Pareto-front that dominate $\bm{z}$, then REPOG can be interpreted as the relative distance of a suboptimal point from the Pareto front.

Given a Pareto front PF which contains only efficient points, but which may not be complete and is finite, and an approximate finite Pareto front $\overline{\text{PF}}$, it can be useful to somehow characterize the suboptimality of the approximate front against the exact front. One possible method is to simply count the number of points in the approximate front which are dominated by the exact front. A robust and conservative measure of suboptimality is to report the maximum REPOG value amongst the points in the approximate front. We can also report the mean of the non-zero REPOG values. Together, the number of suboptimal points, the maximum REPOG, and the mean non-zero REPOG give an insight into how well the approximate front compares to the exact front. For example, a high number of suboptimal points with a low mean and maximum REPOG might suggest a consistently low optimality gap, whereas a low number of suboptimal points with a high mean and even higher maximum REPOG may suggest that the approximate front is generally close to optimality, despite having some significant weak spots. We argue that this type of reporting is more useful than the standard hypervolume metric \citep{shang2020hypervolume_survey}, which relies on the choice of a reference point, is unintuitive, and obfuscates detail.

From \autoref{tab:exactVsApprox} we observe that the exact solution times of the problem instances grow by orders of magnitude with the number of subsystems, whereas the heuristic solution times are all less than 0.5 seconds. On the other hand, the exact method consistently finds many more solutions than the heuristic. This is to be expected, as the heuristic is limited to separable policies. An intuitive explanation is that, although the exact solutions use deterministic policies, one can imagine the policies being ``separable'' but using the states of other subsystems as sources of randomness; therefore, they can ``emulate'' randomised separable policies on each subsystem by using the states of other subsystems as a source of pseudo-randomness. This could allow them to strike a finer balance between the two objectives. This may not be the only explanation, but it does provide some insight. Note that in the exact case we report the number of solutions and the number ``significant'' solutions. The latter refers to the size of a subset of the solutions where there is more than a 0.1\% difference in both objective values.

In terms of error, we note that a small handful of instances (1:2, 2:3, 7:8, 12:13, 13:14) achieved a maximum REPOG value of zero or significantly less than 0.01\%. We note that all of these instances contain only two subsystems. Problem 10:13 exhibits the largest number of suboptimal solutions, as well as the highest proportion of solutions being suboptimal (40\%). However, as suggested by the low mean REPOG values and inspection by eye in \autoref{fig:10-13}, we find that the exact and approximate Pareto fronts are quite similar, and that the extent to which the suboptimal solutions are dominated is minor. The largest maximum REPOG value is 12.4\%, given by Problem 1:3. This particular point is the orange point with an operational cost of about 45 in the respective plot in \autoref{fig:1-3}, and is strongly dominated by an exact solution. However, this particular point is clearly a poor trade-off between two other heuristic solutions on either side, and could be safely ignored by decision-makers.

For instances with five subsystems or more, the exact method becomes computationally infeasible and we are unable to provide comparisons between the heuristic method and optimal solutions. However, we can report the run-times and the numbers of solutions found. The full 14-subsystem problem was solved in 5.22 seconds, with 258 solutions in the approximate Pareto front. All other instances were solved in less time than this, as is to be expected. The full table of results can be found in \autoref{app:table}. Focusing on the full 14-subsystem example, \autoref{fig:fullSystem} shows the approximate Pareto front for the whole system. It can be seen that the 258 solutions are well-spaced over the Pareto front, and that even if the exact front has more solutions, this will not realistically benefit the decision maker in terms of flexibility. The most reliable solution has an operational cost of 372.25, and a log failure rate (LFR) of -3.472, i.e., a failure rate of 0.031. The Pareto front allows us to do some sensitivity analysis, and weigh the benefits of small sacrifices in reliability against the savings in operational costs. For example, if our target failure rate is 0.04 (LFR $\approx 3.22$), we can select a solution with a cost of 335.48 and a failure rate of 0.0398, providing a substantial saving. As another example, if we have a target failure rate of 0.05 (LFR $\approx -3$), we can choose a solution with a cost of 311.42 and a failure rate of 0.049. Of course, it is not possible to know whether or not these solutions are truly Pareto optimal, except for the two that minimise the failure rate and cost rate. However, it is ultimately up to the decision-maker to decide whether or not the trade-offs are worthwhile.

\begin{table}[tbp]
    \centering
    \small
    \begin{tabular}{p{1.1cm}p{1.4cm}p{1.5cm}p{1.6cm}p{1cm}p{1.5cm}p{1.6cm}p{1.3cm}p{1.3cm}}
    \toprule
    & \multicolumn{3}{c}{Exact} & \multicolumn{5}{c}{Heuristic}\\ \cmidrule(lr){2-4} \cmidrule(lr){5-9}
        Problem & Time & Solutions & Significant Solutions & Time & Solutions & Suboptimal Solutions & Max REPOG (\%) & Mean REPOG (\%) \\ \midrule
        01:02 & 14.74 & 28 & 21 & 0.03 & 9 & 0 & 0.00 & 0.00 \\ 
        01:03 & 365.15 & 152 & 63 & 0.11 & 24 & 3 & 12.4 & 4.87 \\ 
        01:04 & 13609.51 & 739 & 183 & 0.37 & 51 & 14 & 6.96 & 1.73 \\ 
        02:03 & 14.41 & 25 & 18 & 0.03 & 11 & 1 & $<0.01$ & $<0.01$ \\ 
        02:04 & 365.49 & 180 & 88 & 0.12 & 28 & 8 & 2.39 & 1.57 \\ 
        02:05 & 11030.04 & 727 & 173 & 0.38 & 55 & 21 & 6.9 & 1.58 \\ 
        03:04 & 26.06 & 52 & 42 & 0.05 & 18 & 3 & 6.07 & 3.41 \\ 
        03:05 & 716.84 & 237 & 116 & 0.25 & 42 & 13 & 9.65 & 3.63 \\ 
        03:06 & 13256.21 & 803 & 157 & 0.36 & 62 & 21 & 3.87 & 1.57 \\ 
        04:05 & 25.34 & 50 & 40 & 0.05 & 19 & 4 & 6.38 & 2.51 \\ 
        04:06 & 357.72 & 174 & 73 & 0.19 & 32 & 6 & 6.1 & 2.89 \\ 
        04:07 & 4515.55 & 485 & 128 & 0.21 & 42 & 13 & 4.91 & 1.53 \\ 
        05:06 & 14.67 & 25 & 17 & 0.03 & 12 & 2 & 6.57 & 3.29 \\ 
        05:07 & 162.16 & 99 & 38 & 0.08 & 22 & 5 & 6.26 & 2.84 \\ 
        05:08 & 10447.3 & 810 & 179 & 0.34 & 58 & 20 & 5.3 & 1.96 \\ 
        06:07 & 8.11 & 16 & 12 & 0.03 & 9 & 2 & 6.27 & 3.13 \\ 
        06:08 & 314.52 & 141 & 71 & 0.12 & 30 & 4 & 4.74 & 3.15 \\ 
        06:09 & 3967.18 & 430 & 102 & 0.26 & 43 & 11 & 2.53 & 1.68 \\ 
        07:08 & 22.63 & 39 & 38 & 0.05 & 14 & 2 & $<0.01$ & $<0.01$ \\ 
        07:09 & 314.79 & 141 & 75 & 0.08 & 25 & 5 & 6.51 & 2.91 \\ 
        07:10 & 11293.63 & 703 & 211 & 0.27 & 41 & 12 & 3.1 & 1.44 \\ 
        08:09 & 23.27 & 38 & 29 & 0.04 & 16 & 3 & 7.53 & 4.49 \\ 
        08:10 & 625.71 & 247 & 125 & 0.14 & 33 & 9 & 6.47 & 2.57 \\ 
        08:11 & 10894.5 & 748 & 202 & 0.33 & 49 & 17 & 4.08 & 1.84 \\ 
        09:10 & 14.38 & 29 & 22 & 0.03 & 12 & 2 & 7.04 & 3.52 \\ 
        09:11 & 160.06 & 100 & 48 & 0.08 & 21 & 3 & 2.72 & 2.48 \\ 
        09:12 & 11093.29 & 728 & 211 & 0.29 & 48 & 18 & 3.7 & 1.68 \\ 
        10:11 & 14.31 & 29 & 23 & 0.04 & 12 & 2 & 8.9 & 4.45 \\ 
        10:12 & 623.01 & 248 & 141 & 0.18 & 36 & 12 & 6.0 & 2.55 \\ 
        10:13 & 10676.22 & 739 & 178 & 0.28 & 65 & 26 & 6.24 & 1.95 \\ 
        11:12 & 22.81 & 38 & 34 & 0.04 & 16 & 3 & 7.44 & 2.48 \\ 
        11:13 & 320.6 & 132 & 52 & 0.11 & 38 & 7 & 8.25 & 5.1 \\ 
        11:14 & 3344.99 & 449 & 84 & 0.24 & 53 & 18 & 8.77 & 3.52 \\ 
        12:13 & 22.59 & 35 & 24 & 0.04 & 20 & 1 & $<0.01$ & $<0.01$ \\ 
        12:14 & 329.11 & 129 & 47 & 0.11 & 37 & 7 & 7.38 & 4.08 \\ 
        13:14 & 8.21 & 14 & 8 & 0.03 & 9 & 0 & 0.00 & 0.00 \\ \bottomrule
    \end{tabular}
    \caption{Comparison between exact and heuristic methods}
    \label{tab:exactVsApprox}
\end{table}
\begin{figure}[tbp]
    \centering
    \includegraphics[width=0.9\linewidth]{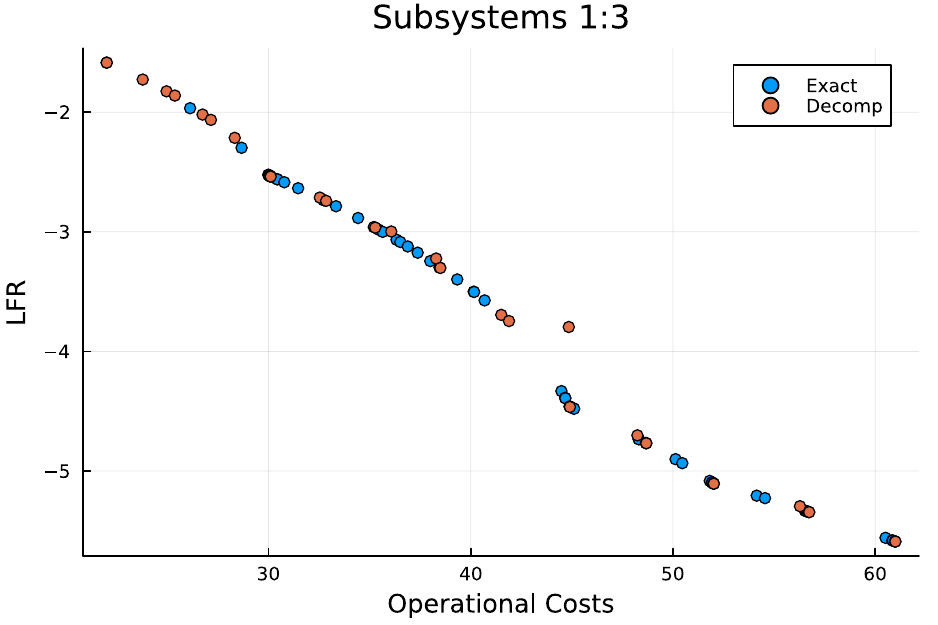}
    \caption{Exact and approximate Pareto fronts for problem 1:3}
    \label{fig:1-3}
    \vspace{1em}
    \includegraphics[width=0.9\linewidth]{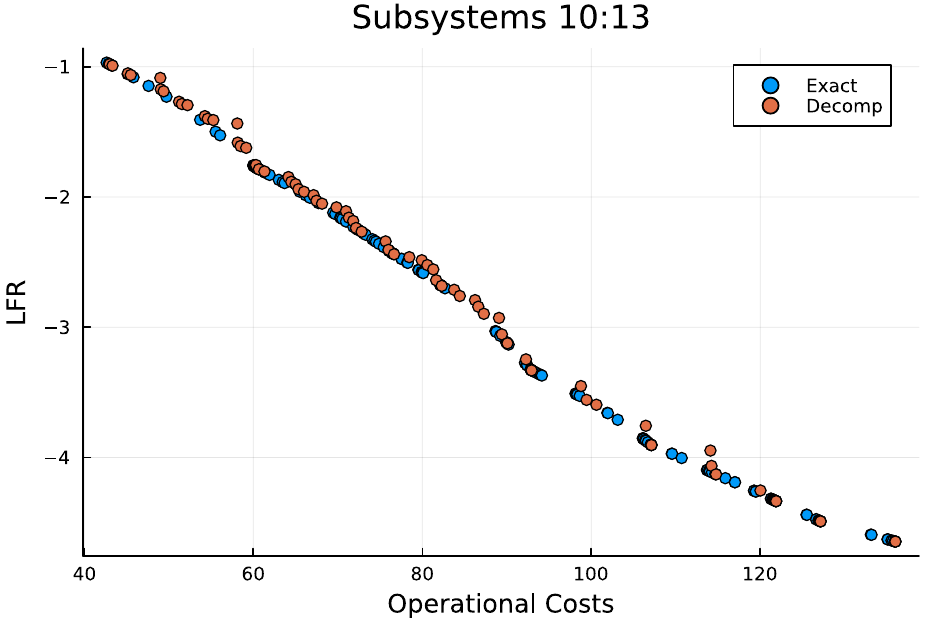}
    \caption{Exact and approximate Pareto fronts for problem 10:13}
    \label{fig:10-13}
\end{figure}
\begin{figure}[tbp]
    \centering
    \includegraphics[width=0.9\linewidth]{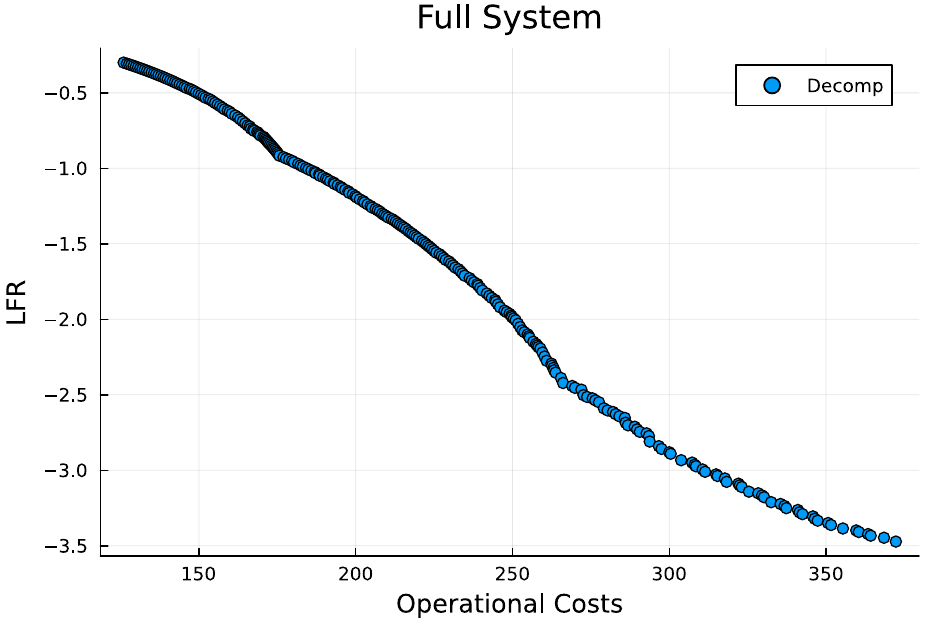}
    \caption{Approximate Pareto front for the full system}
    \label{fig:fullSystem}
\end{figure}

\section{Conclusion}
\label{paper2:conc}
This paper has proposed a simpler alternative to the standard approach for the inclusion of impulsive control in CTMDP models found in the literature, and developed theory using equivalence relations to deal with the symmetries in these models. We have also identified an intuitive class of policies - the coherent policies - which always includes an optimal policy and possesses intuitive properties. These principles have also been extended to large impulsive CTMDPs, which can be regarded as being made up of smaller impulsive CTMDPs, and we have shown how to make use of our equivalence theory to analyse and heuristically optimise these models. We have also applied our modelling framework and heuristic solution methodology to a dynamic maintenance problem, and demonstrated that the heuristic performs well in terms of near-Pareto-optimality and scales much better than exact methods.

Currently, our framework lacks some of the flexibility of pre-existing frameworks of impulsive control for CTMDPs. Specifically, we do not allow for uncertainty in the outcomes of taking impulsive actions, we do not allow for lump-sum costs for taking impulsive actions, and we only allow for one gradual action (i.e., the zero action) per state. Further work could be done to add these features to our framework. However, we argue that the simplicity of our framework, and its compatibility with out-of-the-box solution methods, makes it an attractive option for modellers. Additionally, the framework could be enhanced by considering modelling features not currently present in the impulsive control literature, such as constraints. Specifically, constraints that span the whole system could translate into linking constraints in the block-diagonal problem, necessitating an iterative column generation approach, even in the single objective case. Such a method could be applicable to weakly-coupled MDPs, for which it is common to relax instantaneous action constraints to long-run average or discounted constraints.

Further research could also consider other application areas in which it makes sense to model problems as impulsive or decomposable impulsive CTMDPs. For example, one could consider queueing control problems with many queues that can be treated independently, or inventory problems under continuous review with many products. Where our heuristic methodology is applicable, it could be tested and compared with exact optimal solutions in order to investigate the Pareto-optimality gap.  

\section*{Declarations}

\textbf{Funding and/or Conflicts of interests/Competing interests}\\[-12pt]

\noindent This paper is based on work completed at the EPSRC-funded STOR-i Centre for Doctoral Training (EP/S022252/1). The authors declare that they have no conflicts of interest or competing interests.

\bibliographystyle{apalike}
\bibliography{bib}
\appendix

\section{Proofs for \autoref{paper2:model}}
\subsection{Proof of \autoref{lemma:s-aHat}}
\label{app:lemma:s-aHat}
For any state $s$, we see that:
\begin{align*}
    \widehat{\pi}(\widehat{s,\mu(s)}) &= \sum_{(s',a')\in\widehat{(s,\mu(s))}}\pi(s',a')\\
    &= \sum_{s'\in\mathcal{S}:(s',\mu(s')\in\widehat{(s,\mu(s))}}\pi(s',\mu(s')) \text{ (at most one non-zero term per state)}\\
    &= \sum_{s'\in\mathcal{S}:(s',\mu(s')\in\widehat{(s,\mu(s))}}\pi^\mu(s')\\
    &= \widehat{\pi^\mu}(\widehat{s}).
\end{align*}
Now it's clear that the $\widehat{\pi}(\widehat{s,\mu(s)})$ alone sum to one, so $\widehat{\pi}(\widehat{s,a}) = 0$ for all state action equivalence classes $\widehat{s,a}$ that can't be written in the form $\widehat{s',\mu(s')}$ for some $s'\in\mathcal{S}$.

\subsection{Effect of Sequencing for Suboptimal Policies}
\label{app:boundExample}

\begin{proposition}
\label{prop:sequencedBoundsGetWorse}
    Let $\mu^*$ be an average-optimal policy for an impulsive CTMDP. Let $J^{\mu^*}_t$ be the finite-horizon value function under policy $\mu^*$ for state $s$ until time $t$. Assume $t>T$ for some $T$ large enough that $\mu^*$ is $t$-optimal for all $t>T$, and $T\geq t(s,a)$ for all $s,a$. Let $Q^{\mu^*}_t$ be the action value function. Let $a_1\neq0$ be an action for state $s$ satisfying $Q^{\mu^*}_t(s,a_1)= (1+\varepsilon(t))J^{\mu^*}_t(s)$ for some $\varepsilon(t)\geq0$, and let $a_2\neq0$ be an action for $s\oplus a_1$ satisfying $Q^{\mu^*}_t(s\oplus a_1,a_2) = (1+\varepsilon'(t))J^{\mu^*}_t(s\oplus a_1)$ for some $\varepsilon'(t)\geq0$. Then $Q^{\mu^*}_t(s,a_1\oplus a_2) \leq (1 + \varepsilon(t))(1+\varepsilon'(t))J^{\mu^*}_t(s)$. Additionally, there exist cases where this bound is tight. 
\end{proposition}
\begin{proof}
    We see:
\begin{align*}
    Q^{\mu^*}_t(s,a_1\oplus a_2) &= Q^{\mu^*}_t(s\oplus a_1, a_2) &\text{ (by impulsivity)}\\
    &= (1+\varepsilon'(t))J^{\mu^*}_t(s\oplus a_1) &\text{ (by $\varepsilon'$-bound)}\\
    &\leq (1+\varepsilon'(t))Q^{\mu^*}_t(s\oplus a_1,0) &\text{ (by optimality of $J^{\mu^*}_t$)}\\
    &= (1+\varepsilon'(t))Q^{\mu^*}_t(s,a_1) &\text{ (by impulsivity)} \\
    &= (1+\varepsilon(t))(1+\varepsilon'(t))J^{\mu^*}_t(s) &\text{ (by $\varepsilon$-bound)}.
\end{align*}

We can also demonstrate that this bound is tight. Consider the following impulsive CTMDP:
\begin{align*}
    &\mathcal{S} = \{1,2,3\},\\
    &\mathcal{A}(1) = \{0,1,2\}, \mathcal{A}(2) = \{0,1\}, \mathcal{A}(3) = \{0\},\\
    & q^+(1,0,\cdot) = 0,\, q^+(2,0,s') = \left\{\begin{array}{cc}
         1, & s'=1, \\
         0, & \text{otherwise,}\\
    \end{array}\right. q^+(3,0,s') = \left\{ \begin{array}{cc}
        1, & s' = 1, \\
        0, & \text{otherwise,}
    \end{array} \right.\\
    & c(1) = 1, c(2)>c(1), c(3)>c(2),\\
    & s\oplus a = s+a.
\end{align*}

In words, state 1 is absorbing under the zero action and has the lowest cost rate, states 2 and 3 deterministically transition to state 1 after an expected sojourn time of 1, state 2 is more expensive than state 1, and state 3 is more expensive than state 2. Clearly, the average-optimal policy is $\mu^*(s) = 0$ for all $s$. In fact, this policy is optimal for all finite horizons, so we can write $V^*$ instead of $V^{\mu^*}$ and $Q^*$ instead of $Q^{\mu^*}$. We may evaluate the finite-horizon value functions:
\begin{align*}
    V^*(1;t) &= Q^*(1,0;t) = t\\
    V^*(2;t) &= Q^*(2,0;t)= \begin{cases}
         c(2)t, & t\in[0,1],  \\
         c(2) + (t-1), & t>1. 
    \end{cases}\\
    V^*(3;t) &= Q^*(3,0;t) = \begin{cases}
         c(3)t, & t\in[0,1],  \\
         c(3) + (t-1), & t>1. 
    \end{cases}
\end{align*}
In state 1, we always accumulate cost at rate 1. In state 2, we accumulate cost at rate $c(2)$ until time $t=1$ and then transition to state 1 and accumulate cost at rate 1. State 3 is the same except that $c(2)$ is replaced by $c(3)$. We may also evaluate the remaining action value functions and rewrite them in terms of error functions $\varepsilon(s,a;t)$ where $Q^*(s,a;t) = (1+\varepsilon(s,a;t))V^*(s;t)$, with $\varepsilon(\cdot,\cdot;0) = 0$:
\begin{align*}
    Q^*(1,1;t) &= V^*(2;t) \text{ (by impulsivity)}\\
               &= (1+\varepsilon(1,1;t))V^*(1;t), \text{ for } \varepsilon(1,1;t) = \left\{ \begin{array}{cc}
                    c(2)-1,& t\in(0,1], \\
                    \frac{c(2)-1}{t},& t>1,
               \end{array}\right.\\
    Q^*(1,2;t) &= V^*(3;t) \text{ (by impulsivity)}\\
               &= (1+\varepsilon(1,2;t))V^*(1;t), \text{ for } \varepsilon(1,2;t) = \left\{ \begin{array}{cc}
                    c(3)-1,& t\in(0,1], \\
                    \frac{c(3)-1}{t},& t>1,
               \end{array}\right.\\
    Q^*(2,1;t) &= V^*(3;t) \text{ (by impulsivity)}\\
               &= (1 + \varepsilon(2,1;t))V^*(2;t), \text{ for } \varepsilon(2,1;t) = \left\{ \begin{array}{cc}
                    \frac{c(3)}{c(2)} - 1,& t\in(0,1], \\
                    \frac{c(3)-c(2)}{c(2)-1+t},& t>1. 
               \end{array}\right.
\end{align*}
Some rearranging then shows that $(1+\varepsilon(1,2;t)) = (1+\varepsilon(1,1;t))(1+\varepsilon(2,1;t))$, where action $2$ in state $1$ is the sequenced action of action $1$ in state $1$ and action $1$ in state $2$.
\end{proof}

\autoref{prop:sequencedBoundsGetWorse} provides the rigorous statement and proof of \autoref{claim:doNotSequenceSubOpt}. Note that due to the reliance on $Q^{\mu^*}$, this provides a best-case scenario for the extent of the suboptimality, as it assumes that an optimal policy is followed thereafter. This also extends partially to the multi-objective case via mixed-objective scalarisation. If no DS policy with $\mu(s)=a_1$ and $\mu(s\oplus a_1)=a_2$ is Pareto-optimal, then for every weighting $\bm{w}$ there exists an optimal policy $\mu^*_{\bm{w}}$ such that $a_1$ and $a_2$ are suboptimal in the way described, and therefore $a_1\oplus a_2$ has worse optimality bounds in all of these cases.

\subsection{Commutativity of Projections and Canonical Surjections}
\label{app:ithProjectionsAndCanonicalSurjections}
The following proposition describes the idea that it makes sense to speak of the ``$i\Th$ compartment'' of an impulsive equivalence class $\widehat{(\bm{s},\bm{a})}$, and that this is identical to considering the impulsive equivalence class of the $i\Th$  compartment of any $(\bm{s},\bm{a})\in\widehat{(\bm{s},\bm{a})}$. Mathematically, we can say that assignment to equivalence classes and taking $i\Th$ compartments commute. First, let us recall all necessary notation, and define some additional notation. Consider a stochastically decomposable impulsive CTMDP with feasible state-action space $\mathcal{SA}$, state-action impulsive equivalence relation $\sim_{\oplus}$ and compartment state-action impulsive equivalence relations $\sim_{\oplus_i}$. Let $\widehat{\mathcal{SA}}$ be the partition of $\mathcal{SA}$ induced by $\sim_{\oplus}$, $\mathcal{SA}_i$ be the $i\Th$ compartment state-action space for each $i$, and $\widehat{\mathcal{SA}_i}$ be the partitions of the $\mathcal{SA}_i$ induced by the respective relations $\sim_{\oplus_i}$. Let $\varphi:\mathcal{SA}\to\widehat{\mathcal{SA}}$ be the canonical surjection of $\sim_{\oplus}$, i.e. the function that maps each state-action pair $(\bm{s},\bm{a})$ to its equivalence class, so $\varphi(\bm{s},\bm{a}) = \widehat{(\bm{s},\bm{a})}$. Likewise, let $\varphi_i:\mathcal{SA}_i\to\widehat{\mathcal{SA}_i}$ be the canonical surjections of the $\sim_{\oplus_i}$. Let $\chi_i:\mathcal{SA}\to\mathcal{SA}_i$ be the $i\Th$ projection map of $\mathcal{SA}$ onto $\mathcal{SA}_i$, i.e. $\chi_i(\bm{s},\bm{a})=([\bm{s}]_i,[\bm{a}]_i)$, the function that maps each state-action pair to its $i\Th$ compartment.
\begin{proposition}
     For each $i$, there exists a projection $\psi_i:\widehat{\mathcal{SA}}\to\widehat{\mathcal{SA}_i}$ such that for every $\widehat{(\bm{s},\bm{a})}\in\widehat{\mathcal{SA}}$ and $(\bm{s},\bm{a})\in\widehat{(\bm{s},\bm{a})}$, $\varphi_i([\bm{s}]_i,[\bm{a}]_i) = \psi_i(\widehat{(\bm{s},\bm{a})})$.
\end{proposition}
\begin{proof}
    Let $\widehat{(\bm{s},\bm{a})}\in\widehat{\mathcal{SA}}$, and choose any $(\bm{s}_1,\bm{a}_1),(\bm{s}_2,\bm{a}_2)\in\widehat{(\bm{s},\bm{a})}$. By \autoref{lemma:state-action-equiv}, we have $[\bm{s}_1,\bm{a}_1]_i\sim_{\oplus_i}[\bm{s}_2,\bm{a}_2]_i$ for every $i$, so $\varphi_i([\bm{s}_1,\bm{a}_1]_i) = \varphi_i([\bm{s}_2,\bm{a}_2]_i)$ for every $i$. Therefore any $\widehat{(\bm{s},\bm{a})}\in\widehat{\mathcal{SA}}$ can be mapped to exactly one member of $\widehat{\mathcal{SA}_i}$ for every $i$, and we can denote this by $\psi_i(\widehat{(\bm{s},\bm{a})})$.
\end{proof}
Diagrammatically, we say that the following diagram commutes for all $i$:
\[\begin{tikzcd}[row sep=large,column sep=large]
\mathcal{SA} \arrow{r}{\varphi} \arrow[swap]{d}{\chi_i} & \widehat{\mathcal{SA}} \arrow{d}{\psi_i} \\
\mathcal{SA}_i \arrow{r}{\varphi_i} & \widehat{\mathcal{SA}_i}
\end{tikzcd}\]
The following corollary is equivalent to the above, but for a fixed separable policy $\mu$. Again, we first recall all necessary notation. Consider the whole-state $\mu$-adapted impulsive equivalence relation $\sim_{\oplus\mu}$ and compartment state $\mu$-adapted impulsive equivalence relations $\sim_{\oplus\mu_i}$. Let $\widehat{\mathcal{S}}$ be the partition of $\mathcal{S}$ induced by $\sim_{\oplus\mu}$, and $\widehat{\mathcal{S}_i}$ be the partitions of the $\mathcal{S}_i$ induced by the relations $\sim_{\oplus\mu_i}$. Let $\varphi^\mu:\mathcal{S}\to\widehat{\mathcal{S}}$ be the canonical surjection of $\sim_{\oplus\mu}$, $\varphi(\bm{s}) = \widehat{\bm{s}}$. Likewise, let $\varphi^\mu_i:\mathcal{S}_i\to\widehat{\mathcal{S}_i}$ be the canonical surjections of the $\sim_{\oplus\mu_i}$. Let $\chi^\mu_i:\mathcal{S}\to\mathcal{S}_i$ be the $i$th projection map of $\mathcal{S}$ onto $\mathcal{S}_i$, i.e. $\chi^\mu_i(\bm{s})=[\bm{s}]_i$.
\begin{corollary}
     For each $i$, there exists a projection $\psi^\mu_i:\widehat{\mathcal{S}}\to\widehat{\mathcal{S}_i}$ such that for every $\widehat{\bm{s}}\in\widehat{\mathcal{S}}$ and $\bm{s}\in\widehat{\bm{s}}$, $\varphi^\mu_i([\bm{s}]_i) = \psi^\mu_i(\widehat{\bm{s}})$.
\end{corollary}
\begin{proof}
    Let $\widehat{s}\in\widehat{\mathcal{S}}$, and choose any $\bm{s}_1,\bm{s}_2\in\widehat{\bm{s}}$. By \autoref{coro:state-equiv}, we have $[\bm{s}_1]_i\sim_{\oplus\mu_i}[\bm{s}_2]_i$ for every $i$, so $\varphi^\mu_i([\bm{s}_1]_i) = \varphi^\mu_i([\bm{s}_2]_i)$ for every $i$. Therefore any $\widehat{\bm{s}}\in\widehat{\mathcal{S}}$ can be mapped to exactly one member of $\widehat{\mathcal{S}_i}$ for every $i$, and we can denote this by $\psi^\mu_i(\widehat{\bm{s}})$.
\end{proof}
Diagrammatically, we say that the following diagram commutes for all $i$:
\[\begin{tikzcd}[sep=large]
\mathcal{S} \arrow{r}{\varphi^\mu} \arrow[swap]{d}{\chi^\mu_i} & \widehat{\mathcal{S}} \arrow{d}{\psi^\mu_i} \\
\mathcal{S}_i \arrow{r}{\varphi^\mu_i} & \widehat{\mathcal{S}_i}
\end{tikzcd}\]

\subsection{Proof of \autoref{prop:decompositionWorks}}
\label{app:decompositionWorks}
For any $i$, we want to show that:
    $$\PP(\widehat{S^\mu_{i,k+1}} = \widehat{v}|S^\mu_{k} = \bm{s}) = \PP(\widehat{S^\mu_{i,k+1}} = \widehat{v}|\widehat{S^\mu_{i,k}} = \widehat{[s]_i}),$$
    where $v\in\mathcal{S}_i$. To begin, we see that:
    \begin{align*}
        \PP(S_{i,k+1}^\mu=v|S_{k}^\mu = \bm{s}) &= \PP(S_{k+1}^\mu \in \{\bm{s}'\in\mathcal{S}:[\bm{s}']_i =v\}|S_k^\mu = \bm{s})\\
        &= \sum_{\bm{s}'\in\mathcal{S}:[\bm{s}']_i=v}\PP(S^\mu_{k+1} = \bm{s}'|S_{k}^\mu = \bm{s}).
    \end{align*}
    There are three main cases to consider: (i) the probability of nothing happening, in which case $v = [\bm{s}]_i$; (ii) the probability of something happening in another sub-state, in which case nothing happens to the $i$th sub-state except being updated to reflect the impulsive action, so $v = [\bm{s}\oplus\mu(\bm{s})]_i$, and (iii) the probability that an event happens in compartment $i$. We need only focus on the latter case. We recall that for any state $\bm{s}$, if there is an event in compartment $i$ then all other compartments $j\neq i$ deterministically update, so that $[\bm{s}']_j = [\bm{s}\oplus\mu(\bm{s})]_j$. Hence, for any $v \notin \{[\bm{s}]_i,[\bm{s}\oplus\mu(\bm{s})]_i\}$ we can write:
    \begin{align*}
        \PP(S_{i,k+1}^\mu=v|S_{k}^\mu = \bm{s}) &= \sum_{\bm{s}'\in\mathcal{S}:[\bm{s}']_i=v}\PP(S^\mu_{k+1} = \bm{s}'|S_{k}^\mu = \bm{s})\\
        &= \PP(S_{k+1}^\mu = \bm{s}'' | S_k^\mu = \bm{s})\\
        &= q(\bm{s},\mu(\bm{s}),\bm{s}'')\Delta\\
        &= q_{\mu_i}([\bm{s}]_i,v)\Delta, \text{ (assuming a separable policy),}
    \end{align*}
    where $\bm{s}'' = \left([\bm{s}]_1\oplus_1\mu_1([\bm{s}]_1),\dots,v,\dots,[\bm{s}]_I\oplus_I\mu_I([\bm{s}]_I)\right)$.
    
    We now wish to consider $\widehat{S_{i,k+1}^\mu}$. Choose some equivalence class $\widehat{v} \neq \widehat{[\bm{s}]_i}$, and recognise that by coherence we have that $[\bm{s}]_i\sim_{\oplus\mu_i} [\bm{s}\oplus\mu(\bm{s})]_i$, so $[\bm{s}]_i\notin\widehat{v}$ and  $[\bm{s}\oplus\mu(\bm{s})]_i\notin \widehat{v}$. Then we can see:
    \begin{align*}
        \PP(\widehat{S_{i,k+1}^\mu} = \widehat{v}|S_{k}^\mu = \bm{s}) &= \PP(S_{i,k+1}^\mu \in \widehat{v}|S_{k}^\mu = \bm{s})\\
        &= \sum_{v'\in\widehat{v}}\PP(S_{i,k+1}^\mu = v'|S_{k}^\mu = s)\\
        &= \sum_{v'\in\widehat{v}} q_{\mu_i}([\bm{s}]_i,v')\Delta\\
        &=\sum_{v'\in\widehat{v}} q_{\mu_i}(\widehat{[\bm{s}]_i},v')\Delta\\
        &= q_{\mu_i}(\widehat{[s]_i}, \widehat{v})\Delta.
    \end{align*}
    We note that this expression depends only on the sub-state equivalence classes, so we can therefore write:
    $$\PP(\widehat{S_{i,k+1}^\mu} = \widehat{v}'| \widehat{S_{i,k}^\mu} = \widehat{v}) = q_{\mu_i}(\widehat{v}, \widehat{v}')\Delta,$$
    for any $\widehat{v}'\neq \widehat{v}.$ We can then see that:
    \begin{align*}
        \PP(\widehat{S_{i,k+1}^\mu} = \widehat{v}| \widehat{S_{i,k}^\mu} = \widehat{v}) &= 1 - \PP(\widehat{S_{i,k+1}^\mu} \neq \widehat{v}| \widehat{S_{i,k}^\mu} = \widehat{v})\\
        &= 1 - \sum_{\widehat{v'}\neq \widehat{v}} \PP(\widehat{S_{i,k+1}^\mu} = \widehat{v}'| \widehat{S_{i,k}^\mu} = \widehat{v})\\
        &= 1 - \sum_{\widehat{v}'\neq \widehat{v}} q_{\mu_i}(\widehat{v},\widehat{v}') \Delta\\
        &= 1 + q_{\mu_i}(\widehat{v},\widehat{v})\Delta.
    \end{align*}
    Therefore $\PP(\widehat{S_{i,k+1}^\mu} = \widehat{v}|{S_{k}^\mu} = \bm{s})$ depends only on the respective compartment equivalence class of $\bm{s}$, so the stochastic processes $\lcurl\widehat{S_{i,k}^\mu}\rcurl_{k=1}^\infty$ for each $i$ are independent Markov chains.
    
\subsection{Proof of \autoref{prop:addMultOR_are_decompable}}
\label{app:addMultLogicalORAreDecomp}
Let $\bm{g}:[0,1]^{|\mathcal{SA}|}\to\RR^K, \bm{g}(\bm{\pi}) = \sum_{\bm{s}\in\mathcal{S}}\bm{c}(\bm{s}\oplus \bm{a})\pi(\bm{s},\bm{a})$ be the function mapping every solution of OP to its long-run average cost. First, assume $c_k$ is an additive cost function. We see:
    \begin{align*}
        g_k(\bm{\pi}) &= \sum_{(\bm{s},\bm{a})\in\mathcal{SA}}c_k(\bm{s}\oplus \bm{a})\pi(\bm{s},\bm{a})\\
        &= \sum_{(\bm{s},\bm{a})\in\mathcal{SA}}\left(\sum_{i=1}^I c_{i,k}([\bm{s}\oplus \bm{a}]_i)\right)\pi(\bm{s},\bm{a})\\
        &= \sum_{i=1}^I\sum_{(\bm{s},\bm{a})\in\mathcal{SA}}c_{i,k}([\bm{s}\oplus \bm{a}]_i)\pi(\bm{s},\bm{a})\\
        &=\sum_{i=1}^I \sum_{(s_i,a_i)\in\mathcal{SA}_i}\sum_{(\bm{s},\bm{a})\in\mathcal{SA}:[\bm{s},\bm{a}]_i=(s_i,a_i)}c_{i,k}(s_i\oplus a_i)\pi(\bm{s},\bm{a})\\
        &=\sum_{i=1}^I\sum_{(s_i,a_i)\in\mathcal{SA}_i}c_{i,k}(s_i\oplus a_i)\pi_i(s_i,a_i).
    \end{align*}
    Therefore $c_k$ is decomposable with $f_k:\RR^I\to\RR$, $f_k(\bm x) = \sum_{i=1}^I x_i$. Next, assume $c_k$ is a multiplicative cost function. We see:
    \begin{align*}
        g_k(\pi) &= \sum_{(\bm{s},\bm{a})\in\mathcal{SA}}c_k(\bm{s}\oplus \bm{a})\pi(\bm{s},\bm{a})\\
        &= \sum_{\widehat{\bm{s},\bm{a}}\in\widehat{\mathcal{SA}}}c_k(\widehat{\bm{s},\bm{a}})\widehat{\pi}(\widehat{\bm{s},\bm{a}})\\
        &=\sum_{\widehat{\bm{s},\bm{a}}\in\widehat{\mathcal{SA}}}\left(\prod_{i=1}^Ic_{i,k}([\widehat{\bm{s},\bm{a}}]_i)\right)\left(\prod_{i=1}^I \widehat{\pi}_i([\widehat{\bm{s},\bm{a}}]_i)\right)\quad \\&\quad\quad\text{(by independence of sub-processes under separable $\mu$)}\\
        &= \sum_{\widehat{\bm{s},\bm{a}}\in\widehat{\mathcal{SA}}}\prod_{i=1}^Ic_{i,k}([\widehat{\bm{s},\bm{a}}]_i) \widehat{\pi}_i([\widehat{\bm{s},\bm{a}}]_i)\\
        &= \prod_{i=1}^I \sum_{\widehat{s_i,a_i}\in\widehat{SA}_i}c_{i,k}(\widehat{s_i,a_i})\pi_i(\widehat{s_i,a_i})\\
        &=\prod_{i=1}^I\sum_{s_i,a_i\in\mathcal{SA}_i}c_{i,k}(s_i,a_i)\pi_i(s_i,a_i).
    \end{align*}
    Therefore $c_k$ is decomposable with $f_k:\RR^I\to\RR$, $f_k(\bm x) = \prod_{i=1}^I x_i$.  Finally, assume $c_k$ is a logical-OR cost function. We first see that:
    \begin{align*}
        c_k(\bm{s}) &= \II\lcurl\bigvee_{i=1}^I ([\bm{s}]_i\in E_i)\rcurl\\
        &= \II\lcurl\neg\neg\bigvee_{i=1}^I ([\bm{s}]_i\in E_i)\rcurl\\
        &= 1 - \II\lcurl\bigwedge_{i=1}^I ([\bm{s}]_i\notin E_i)\rcurl\\
        &= 1 - \prod_{i=1}^I\II\lcurl[\bm{s}]_i\notin E_i\rcurl.
    \end{align*}
    We then see that:
    \begin{align*}
        g_k(\bm{\pi}) &= \sum_{(\bm{s},\bm{a})\in\mathcal{SA}}c_k(\bm{s}\oplus \bm{a})\pi(\bm{s},\bm{a})\\
        &= \sum_{(\bm{s},\bm{a})\in\mathcal{SA}}\left(1 - \prod_{i=1}^I\II\lcurl[\bm{s}]_i\notin E_i\rcurl\right)\pi(\bm{s},\bm{a})\\
        &= 1 - \sum_{(\bm{s},\bm{a})\in\mathcal{SA}}\left(\prod_{i=1}^I\II\lcurl[\bm{s}]_i\notin E_i\rcurl\right)\pi(\bm{s},\bm{a}).
    \end{align*}
    We then treat the sum along similar lines to the multiplicative cost function to find that:
    \begin{align*}
        g_k(\pi) &= 1 - \prod_{i=1}^I\sum_{s_i,a_i\in\mathcal{SA}_i}\II\lcurl s_i\notin E_i \rcurl\pi_i(s_i,a_i)\\
        &= 1 - \prod_{i=1}^I\left(1 - \sum_{s_i,a_i\in\mathcal{SA}_i}\II\lcurl s_i\in E_i\rcurl\pi_i(s_i,a_i)\right),
    \end{align*}
    so $c_k$ is decomposable with $c_{i,k}(s_i) = \II\lcurl s_i\in E_i\rcurl$ and $f_k:[0,1]^I\to[0,1]$, $f(\bm{x}) = 1 - \prod_{i=1}^I (1 - x_i)$. 
    
\section{Proof of \autoref{thm:CS-MILP_and_DC-MIP_equiv}}
\label{app:CS-MILP_and_DC-MIP}
We first show that it is possible to move from CS-MILP solutions to DC-MIP solutions in such a way that equations \eqref{eq-whole-to-compartment-decision-vars} and \eqref{eq-compartment-to-whole-decision-vars} are satisfied, where equations \eqref{eq-whole-to-compartment-decision-vars} and \eqref{eq-compartment-to-whole-decision-vars} are the decision variable equivalent of \autoref{coro:stationaryIndep}.
\begin{proposition}
\label{prop:CS-MILP_to_DC-MIP}
    Let $\mathfrak{C}$ be a decomposable impulsive CTMDP, and choose a decomposition into compartments $\{\mathfrak{C}_i\}_{i=1}^I$. Let $(\bm{\pi},\bm{y})$ be a solution to CS-MILP for $\mathfrak{C}$. Then there exists a solution $\{\bm{\pi}_i,\bm{y}_i\}_{i=1}^I$ for DC-MIP with compartments $\{\mathfrak{C}_i\}_{i=1}^I$ satisfying the equations \eqref{eq-whole-to-compartment-decision-vars} and \eqref{eq-compartment-to-whole-decision-vars}.
\end{proposition}
\begin{proof}
$\pi(\bm{s},\bm{a})$ is a basic feasible solution of OP for $\mathfrak{C}$, and corresponds to a deterministic policy $\mu(\bm{s}) = \arg\max_{\bm{a}\in\mathcal{A}(\bm{s})}\pi(\bm{s},\bm{a})$. Moreover, by the separable and coherent constraints, $\mu$ must be a separable and coherent policy. Therefore, $\mu$ can be written as $(\mu_1,\dots,\mu_I)$ where the $\mu_i$ are coherent policies for the compartments $\mathfrak{C}_i$. Now, let $V^{\mu_i}$ be the stationary distribution of the CTMC defined by compartment $\mathfrak{C}_i$ as controlled by policy $\mu_i$, for each $i$. Then: $$\pi_i(v,b) := \begin{cases}
     \PP(V^{\mu_i} = v),& b = \mu_i(v)  \\
     0, & otherwise, 
\end{cases}.$$
provides a basic feasible solution for OP for $\mathcal{C}_i$, satisfies the equivalent constraints in DC-MIP, and satisfies \eqref{eq-whole-to-compartment-decision-vars} and \eqref{eq-compartment-to-whole-decision-vars}. If we let $y_i(v,b) = \II\{b=\mu_i(v)\}$ for all $i$, and recall that the $\mu_i$ are coherent, then $\{\bm{\pi}_i,\bm{y}_i\}_{i=1}^I$ provides a feasible solution to DC-MIP.
\end{proof}
The next result shows that the same approach works in the other direction, moving from solutions of DC-MIP to solutions of CS-MILP.
\begin{proposition}
\label{prop:DC-MIP_to_CS-MILP}
Let $\mathfrak{C}$ be a decomposable impulsive CTMDP, and choose a decomposition into compartments $\{\mathfrak{C}_i\}_{i=1}^I$. Let $\{\bm{\pi}_i,\bm{y}_i\}_{i=1}^I$ be a solution to DC-MIP for compartments $\{\mathfrak{C}_i\}_{i=1}^I$. Then there exists a solution $(\bm{\pi},\bm{y})$ to CS-MILP for $\mathfrak{C}$ satisfying \eqref{eq-whole-to-compartment-decision-vars} and \eqref{eq-compartment-to-whole-decision-vars}.
\end{proposition}
\begin{proof}
The $\bm{\pi}_i$ correspond to basic feasible solutions of OP for each $\mathfrak{C}_i$, which in turn corresponds to coherent policies $\mu_i(v)=\arg\max_{b\in\mathcal{A}_i(v)}\pi_i(v,b)$. Together, these form the CS policy $\mu=(\mu_1,\dots,\mu_I)$ for the whole CTMDP $\mathfrak{C}$. Let $S^\mu$ be the stationary distribution of $\mathfrak{C}$ controlled by $\mu$. Now we define:
$$\pi(\bm{s},\bm{a}):=\begin{cases}
\PP(S^\mu = \bm{s}),& \bm{a} = \mu(\bm{s}),\\ 
0,& \text{otherwise}.\end{cases}
$$
Then $\pi(\bm{s},\bm{a})$ provides a basic feasible solution to OP. Together with $y(\bm{s},\bm{a}):=\II\{\bm{a}=\mu(\bm{s})\}$, $(\bm{\pi},\bm{y})$ gives a solution to CS-MILP that satisfies the relevant equations.
\end{proof}
As a result, we can move from solutions of CS-MILP to DC-MIP and vice versa in such a way that the relevant equations are satisfied. We now bring these together to prove \autoref{thm:CS-MILP_and_DC-MIP_equiv}.

\begin{proof}

We must show that the existence of a solution with a certain objective value in one of the problems implies the existence of a solution in the other problem with the same objective value.

First, let $(\bm{\pi},\bm{y})$ be a solution to CS-MILP. We want to show that it has a corresponding solution in DC-MIP with the same objective value. The solution $(\bm{\pi},\bm{y})$ can be mapped to a solution $(\bm{\pi}_i,\bm{y}_i)_{i=1}^I$ of DC-MIP by \autoref{prop:CS-MILP_to_DC-MIP}, with $\bm{\pi}$ and $\bm{\pi}_i$ related by \eqref{eq-whole-to-compartment-decision-vars} and \eqref{eq-compartment-to-whole-decision-vars}. We recall:
$$g_k(\bm{\pi}) = f_k\left(\left(\sum_{s_1,a_1\in\mathcal{SA}_1}c_{1,k}(s_1\oplus a_1)\pi_1(s_1,a_1)\right)_{i=1}^I\right),$$
and we need only focus on one entry in this, so for any compartment $i$ and objective $k$ we have:
\begin{align*}
\sum_{(s_i,a_i)\in\mathcal{SA}_i}c_{i,k}(s_i\oplus a_i)\pi_i(s_i,a_i) 
&= \sum_{(s_i,a_i)\in\mathcal{SA}_i}\sum_{(\bm{s},\bm{a})\in\mathcal{SA}:[\bm{s},\bm{a}]_i=(s_i,a_i)}c_{i,k}(s_i\oplus a_i)\pi(\bm{s},\bm{a})\\
&= \sum_{\widehat{s_i,a_i}\in\widehat{\mathcal{SA}_i}}\sum_{(s_i,a_i)\in\widehat{s_i,a_i}}\sum_{(\bm{s},\bm{a})\in\mathcal{SA}:[\bm{s},\bm{a}]_i=(s_i,a_i)}c_{i,k}(s_i\oplus a_i)\pi(\bm{s},\bm{a})\\
&=\sum_{\widehat{s_i,a_i}\in\widehat{\mathcal{SA}_i}}c_{i,k}(\widehat{s_i,a_i})\sum_{(\bm{s},\bm{a})\in\mathcal{SA}:[\bm{s},\bm{a}]_i\in\widehat{s_i,a_i}}\pi(\bm{s},\bm{a})\\
&=\sum_{\widehat{s_i,a_i}\in\widehat{\mathcal{SA}_i}}c_{i,k}(\widehat{s_i,a_i})\sum_{\widehat{\bm{s},\bm{a}}\in\widehat{\mathcal{SA}}}\sum_{\bm{s},\bm{a}\in\widehat{s,a}:[\bm{s},\bm{a}]_i\in\widehat{s_i,a_i}}\pi(\bm{s},\bm{a})\\
&=\sum_{\widehat{s_i,a_i}\in\widehat{\mathcal{SA}_i}}c_{i,k}(\widehat{s_i,a_i})\sum_{\widehat{\bm{s},\bm{a}}\in\widehat{\mathcal{SA}}:[\widehat{\bm{s},\bm{a}}]_i=\widehat{s_i,a_i}}\widehat{\pi}(\widehat{\bm{s},\bm{a}})\\
&=\sum_{\widehat{s_i,a_i}\in\widehat{\mathcal{SA}_i}}c_{i,k}(\widehat{s_i,a_i})\widehat{\pi}_i(\widehat{s_i,a_i}).
\end{align*}
Therefore $\bm{g}(\bm{\pi})$ is entirely determined by the $\pi_i(\widehat{s_i,a_i})$ as determined by the equivalent solution to DC-MIP, so every solution to CS-MILP has a solution in DC-MIP with the same objective value.  

Secondly, we want to show the other direction. Let $(\bm{\pi}_i,\bm{y}_i)_{i=1}^I$ be a solution to DC-MIP. We want to show that it has a corresponding solution in CS-MILP with the same objective value. The solution $(\bm{\pi}_i,\bm{y}_i)_{i=1}^I$ can be mapped to a solution $(\bm{\pi},\bm{y})$ of CS-MILP by \autoref{prop:DC-MIP_to_CS-MILP}, with $\bm{\pi}$ and $\bm{\pi}_i$ related by \eqref{eq-whole-to-compartment-decision-vars} and \eqref{eq-compartment-to-whole-decision-vars}. Starting with the objective functions of DC-MIP, with $\bm{\pi}_i$ being the decision variables of DC-MIP (not the sum over variables of CS-MILP), we have:
    \begin{align*}
        \sum_{s_i,a_i\in\mathcal{SA}_i}c_{i,k}(s_i\oplus a_i)\pi_i(s_i,a_i) &= \sum_{\widehat{s_i,a_i}\in\widehat{\mathcal{SA}_i}}c_{i,k}(\widehat{s_i,a_i})\sum_{s_i,a_i\in\widehat{s_i,a_i}}\pi_i(s_i,a_i)\\
        &= \sum_{\widehat{s_i,a_i}\in\widehat{\mathcal{SA}_i}}c_{i,k}(\widehat{s_i,a_i})\widehat{\pi}_i(\widehat{s_i,a_i}),
    \end{align*}
    and we can then use the chain of equalities from the first part to get back to the objective function of CS-MILP, where the $\widehat{\pi}_i(\widehat{s_i,a_i})$ are determined by the $\widehat{\pi}(\widehat{\bm{s},\bm{a}})$, which in turn are determined by the equivalent solution to CS-MILP. Therefore every solution to DC-MIP has a solution in CS-MILP with the same objective value, and therefore the bi-implication holds.
\end{proof}

\newpage
\section{Table of results}

\autoref{tab:largeInstances} is discussed in \autoref{paper2:app:results}.

\label{app:table}
\begin{table}[h]
    \begin{minipage}{0.5\textwidth}
    \centering
    \begin{tabular}{lll}
    \toprule
        Problem & Times & Solutions \\ \midrule
        1:5 & 1.04 & 91 \\ 
        1:6 & 1.44 & 122 \\ 
        1:7 & 2.06 & 141 \\ 
        1:8 & 2.84 & 180 \\ 
        1:9 & 3.15 & 188 \\ 
        1:10 & 3.52 & 198 \\ 
        1:11 & 3.76 & 201 \\ 
        1:12 & 4.46 & 237 \\ 
        1:13 & 4.94 & 255 \\ 
        1:14 & 5.22 & 258 \\ 
        2:6 & 0.89 & 84 \\ 
        2:7 & 1.12 & 107 \\ 
        2:8 & 2.03 & 156 \\ 
        2:9 & 3.03 & 165 \\ 
        2:10 & 3.19 & 175 \\ 
        2:11 & 3.31 & 180 \\ 
        2:12 & 3.79 & 220 \\ 
        2:13 & 4.27 & 237 \\ 
        2:14 & 4.55 & 242 \\ 
        3:7 & 0.75 & 78 \\ 
        3:8 & 1.52 & 133 \\ 
        3:9 & 2.06 & 144 \\ 
        3:10 & 2.8 & 158 \\ 
        3:11 & 2.93 & 161 \\ 
        3:12 & 3.75 & 204 \\ 
        3:13 & 4.02 & 234 \\ 
        3:14 & 4.13 & 237 \\ 
        4:8 & 0.58 & 81 \\
         \bottomrule
    \end{tabular}
    \end{minipage}%
    \begin{minipage}{0.5\textwidth}
    \begin{tabular}{lll}
        \toprule
        Problem & Times & Solutions \\ \midrule         
        4:9 & 0.86 & 90 \\ 
        4:10 & 1.66 & 115 \\ 
        4:11 & 2.0 & 128 \\ 
        4:12 & 3.09 & 176 \\ 
        4:13 & 3.71 & 216 \\ 
        4:14 & 3.63 & 217 \\ 
        5:9 & 0.61 & 69 \\ 
        5:10 & 0.95 & 90 \\ 
        5:11 & 1.44 & 112 \\ 
        5:12 & 2.57 & 157 \\ 
        5:13 & 3.01 & 195 \\ 
        5:14 & 3.29 & 198 \\ 
        6:10 & 0.66 & 66 \\ 
        6:11 & 0.88 & 86 \\ 
        6:12 & 1.89 & 128 \\ 
        6:13 & 2.45 & 176 \\ 
        6:14 & 2.97 & 179 \\ 
        7:11 & 0.51 & 61 \\ 
        7:12 & 1.16 & 108 \\ 
        7:13 & 1.95 & 153 \\ 
        7:14 & 2.31 & 161 \\ 
        8:12 & 0.91 & 95 \\ 
        8:13 & 1.43 & 140 \\ 
        8:14 & 2.23 & 153 \\ 
        9:13 & 0.47 & 87 \\ 
        9:14 & 0.89 & 97 \\ 
        10:14 & 0.6 & 80 \\ 
        &&\\
        \bottomrule
    \end{tabular}
    \end{minipage}
    \caption{Heuristic results for large instances}
    \label{tab:largeInstances}
\end{table}

\clearpage
\includepdf[pages=-]{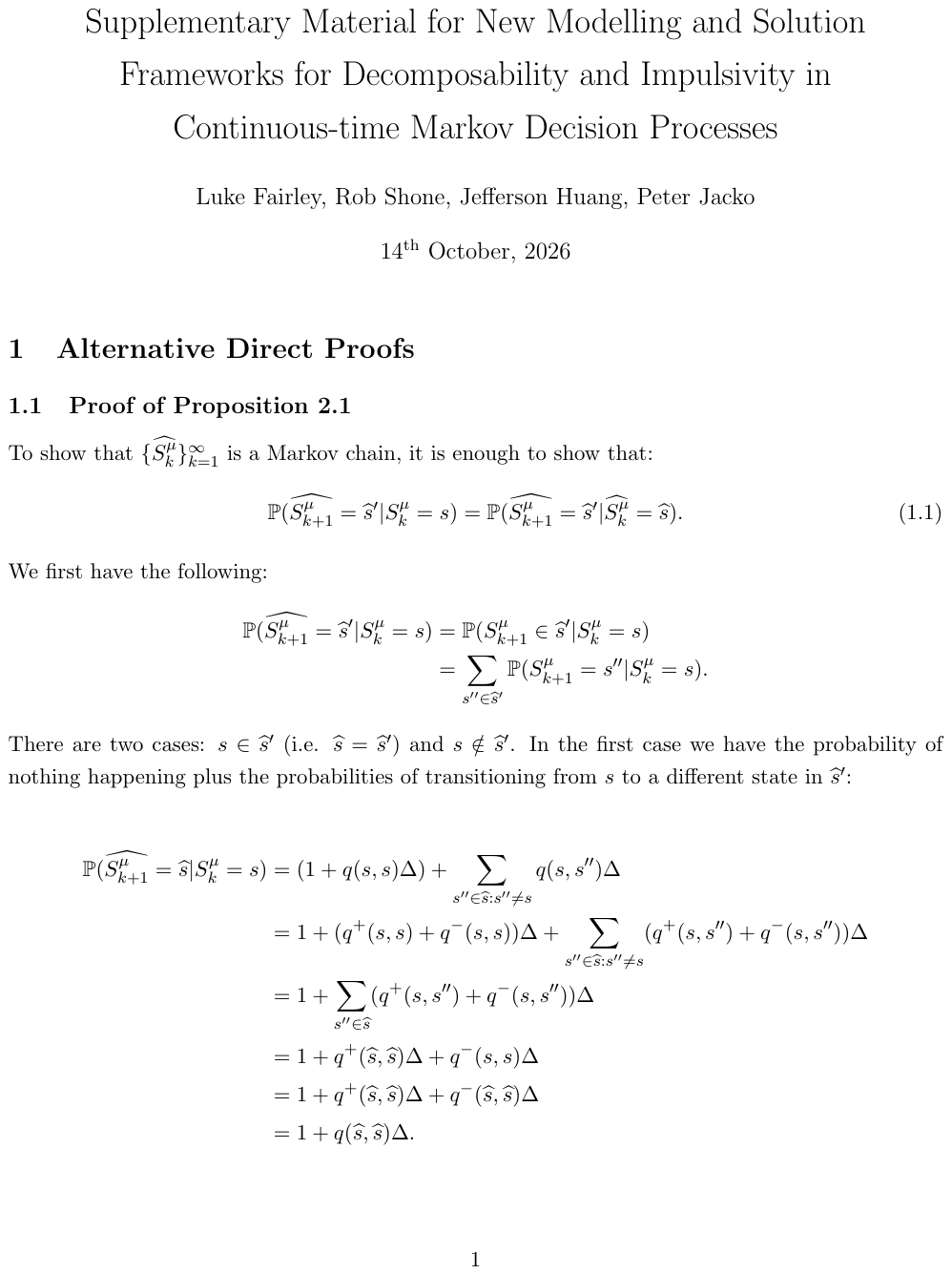}

\end{document}